\documentclass[11pt]{article}

\usepackage[utf8]{inputenc}

\usepackage{
amsmath,amssymb,amsfonts,amsthm}
\usepackage{bm}
\usepackage{colonequals}
\usepackage{comment}
\usepackage{epsfig} \usepackage{latexsym,nicefrac,bbm}
\usepackage
{xspace}
\usepackage{color,fancybox,graphicx,subfigure,fullpage}
\usepackage[top=1.25in, bottom=1.25in, left=1in, right=1in]{geometry}
\usepackage{tabularx}
\usepackage{hyperref}
\usepackage{cleveref}
\usepackage[boxruled,linesnumbered]{algorithm2e}
\usepackage{siunitx}
\usepackage{mathtools}

\usepackage{mathrsfs}

\usepackage{multicol}
\usepackage{enumitem}
\usepackage{float}
\usepackage{tikz}
\usepackage{framed}

\makeatletter
\newcommand{\vast}{\bBigg@{4}}
\newcommand{\Vast}{\bBigg@{5}}
\makeatother

\newtheorem{theorem}{Theorem}[section]

\newtheorem{definition}[theorem]{Definition}
\newtheorem{lemma}[theorem]{Lemma}
\newtheorem{remark}[theorem]{Remark}
\newtheorem{proposition}[theorem]{Proposition}
\newtheorem{corollary}[theorem]{Corollary}

\newcommand{\bC}{\mathbb{C}}

\newcommand{\bN}{\mathbb{N}}

\newcommand{\bR}{\mathbb{R}}

\DeclareSymbolFont{bbold}{U}{bbold}{m}{n}
\DeclareSymbolFontAlphabet{\mathbbold}{bbold}

\newcommand{\spec}{\mathrm{spec}}

\DeclareMathOperator{\diag}{diag}
\DeclareMathOperator{\maxroot}{max\,root}
\DeclareMathOperator{\minroot}{min\,root}
\DeclareMathOperator{\maxsupp}{max\,supp}
\DeclareMathOperator{\minsupp}{min\,supp}

\renewcommand{\epsilon}{\varepsilon}

\newcommand{\SU}{\mathrm{SU}}

\newcommand{\Sym}{\mathrm{Sym}}

\newcommand{\per}{\mathrm{per}}

\newcommand\numberthis{\addtocounter{equation}{1}\tag{\theequation}}

\newcommand\oone{\boldsymbol{1}}

\def\calB{\mathcal{B}}

\def\calH{\mathcal{H}}

\def\calK{\mathcal{K}}

\def\calP{\mathcal{P}}

\def\calR{\mathcal{R}}
\def\calS{\mathcal{S}}

\def\calW{\mathcal{W}}

\renewcommand\AA{\bm{A}}

\newcommand\CC{\bm{C}}

\newcommand\II{\bm{I}}
\newcommand\JJ{\bm{J}}

\newcommand\TT{\bm{T}}
\newcommand\UU{\bm{U}}
\newcommand\VV{\bm{V}}

\def\aa{\bm{a}}
\newcommand\bb{\bm{b}}

\newcommand\uu{\bm{u}}
\newcommand\vv{\bm{v}}

\newcommand\aalpha{\boldsymbol{\alpha}}
\newcommand\bbeta{\boldsymbol{\beta}}

\usepackage{authblk}

\title{Inequalities for rank-two permanents and finite free convolutions\thanks{This research was partially supported by 
	ONR Award N00014-24-12611.}}
\author[1]{Dmitriy Kunisky\thanks{Email: \texttt{kunisky@jhu.edu}.}}
\author[2,3]{Daniel A.\ Spielman\thanks{Email: \texttt{daniel.spielman@yale.edu}.}}
\author[2]{Xifan Yu\thanks{Email: \texttt{xifan.yu@yale.edu}.}}
\affil[1]{Department of Applied Mathematics \& Statistics, Johns Hopkins University}
\affil[2]{Department of Computer Science, Yale University}
\affil[3]{Department of Statistics and Data Science, Yale University}

\date{August 28, 2026}

\begin{document}

\pagenumbering{roman}

\maketitle

\thispagestyle{empty}

\begin{abstract}
    Bang (1976) proved the inequality for matrix permanents $\per^2(\bm A) \geq 2^{-2n} \per(\bm A \otimes \JJ_2)$, where $\bm J_2$ is the $2 \times 2$ all-ones matrix and $\bm A$ is any $n \times n$ matrix with non-negative entries.
    We show that, if $\bm A$ is any $n \times n$ real-valued matrix with rank at most two (possibly having negative entries), this inequality can be sharpened, replacing the constant $2^{-2n}$ by $1 / \binom{2n}{n} = (n!)^2 / (2n)! > 2^{-2n}$.
    We then show that this sharpened inequality also implies new inequalities for finite free convolutions of polynomials: if $p$ and $q$ are monic real-rooted polynomials of degree $n$, then $(p \boxplus_n q)(x)^2 \geq (p^2 \boxplus_{2n} q^2)(x)$ and $(p \boxtimes_n q)(x)^2 \geq (p^2 \boxtimes_{2n} q^2)(x)$ for all $x \in \bR$, for $\boxplus_n$ and $\boxtimes_n$ the finite free additive and multiplicative convolution operations, respectively, on polynomials of degree $n$.
\end{abstract}

\clearpage

\pagestyle{empty}

\tableofcontents

\clearpage

\pagestyle{plain}
\pagenumbering{arabic}

\allowdisplaybreaks{

\section{Introduction}

\subsection{Permanent Inequalities}

Denoting by $\calS_n$ the symmetric group of permutations of $[n] \colonequals \{1, \dots, n\}$, the \emph{permanent} of a matrix $\AA \in \bR^{n \times n}$ is the polynomial in its entries
\[ \per(\AA) \colonequals \sum_{\pi \in \calS_n} \prod_{i = 1}^n \AA_{i,\pi(i)}. \]
Let us write $\bm J_n$ for the $n \times n$ matrix all of whose entries equal 1 and $\bm I_n$ for the $n \times n$ identity matrix.
Van der Waerden's well-known conjecture on permanents \cite{Waerden-1926-Aufgabe45}, later proved independently by Egorychev \cite{Egorychev-1981-SolutionWaerdenPermanents} and Falikman \cite{Falikman-1981-ProofWaerdenConjecture}, states that the smallest possible permanent of a doubly stochastic matrix is $n! / n^n$, achieved by the matrix $\frac{1}{n}\bm J_n$.
Such extremal problems have immediate consequences in combinatorics, and are also relevant to the study of the thermodynamics of the bipartite monomer-dimer model of statistical physics.
See \cite{Friedland-2012-MatchingsRegularGraphs} for a survey.

A preliminary result towards establishing van der Waerden's conjecture was the looser lower bound $\exp(-n)$, established independently by Bang and Friedland \cite{Bang-1979-MatrixFunctionsPermanentConjecture,Friedland-1979-LowerBoundPermanent}.
These proofs are based on the following permanent inequality, credited to Bang's earlier work \cite{Bang-1976-MatrixfunktionerPermanenthypotese}.
\begin{theorem}[Bang's permanent inequality]
    \label{thm:bang}
    For any $\bm A \in \bR^{n \times n}$ with $\AA_{i,j} \geq 0$ for all $i, j \in [n]$,
    \[ \per^2(\bm A) \geq \frac{1}{2^{2n}} \per( \bm A \otimes \JJ_2). \]
\end{theorem}
\noindent
A slightly more symmetric form of the result is obtained by substituting 
$\per^2(\AA) = \per( \bm A \otimes \II_2)$.

Our main technical contribution is the following sharpened version of Bang's inequality in the case where $\AA$ has low rank but is also allowed to have negative entries.
\begin{theorem}\label{thm:perm-inequality}
    Let $\AA \in \bR^{n \times n}$ be a matrix of rank at most $2$. Then,
    \begin{align*}
        \per^2(\AA) \geq \frac{1}{\binom{2n}{n}} \per(\bm A \otimes \JJ_2).
    \end{align*}
\end{theorem}
\noindent
We give the proof of Theorem~\ref{thm:perm-inequality} in Section~\ref{sec:perm-SoS}.
We derive it from the stronger statement, 
 Theorem~\ref{thm:perm-inequality-SoS}, that, for $\AA = \UU\VV^\top$ where $\UU, \VV \in \bR^{n \times 2}$, the difference
\begin{align*}
    \per^2(\AA) - \frac{1}{\binom{2n}{n}} \per(\AA \otimes \JJ_2)
\end{align*}
is an explicit sum of squares of polynomials in the entries of $\UU$ and $\VV$.
As we discuss in Section~\ref{sec:proof-techniques}, our proof  originated from a numerical search for a representation of the difference as a sum of squares.

\begin{remark}
    The assumption on the rank of $\AA$ in Theorem~\ref{thm:perm-inequality} is necessary: the inequality can already fail for $\AA$ of rank at least $3$. For example, it does not hold for the rank-three matrices
    \begin{align*}
        \AA = \begin{bmatrix*}[r]
             -1&-1&-1\\
  -1&-1&1\\
  -1&1&0
        \end{bmatrix*} \,\,\, \text{ or } \,\,\,\, \AA =   \begin{bmatrix*}
  1&1&0&0\\
  0&1&1&0\\
  0&0&1&1\\
  1&0&0&1
  \end{bmatrix*}.
    \end{align*}
    The first example also shows that the non-negativity assumption in Theorem~\ref{thm:bang} is necessary if the assumption on the rank of $\bm A$ is not included.
\end{remark}

\subsection{Finite Free Convolutions}

Our next results draw a surprising connection between Theorem~\ref{thm:perm-inequality} and the \emph{finite free convolution} operations developed in a non-asymptotic version of free probability theory.
The definitions of these operations stem from the work of Marcus, the second author, and Srivastava on interlacing families of polynomials and their applications \cite{MSS-2015-InterlacingFamiliesI,MSS-2015-InterlacingFamiliesII,MSS-2022-InterlacingFamiliesIII,MSS-2018-InterlacingFamiliesIV}, and later were developed into an independent theory including many features parallel to classical free probability \cite{MSS-2022-FiniteFreeConvolution,Marcus-2021-PolynomialConvolutionsFiniteFreeProbability,AP-2018-CumulantsFiniteFreeConvolution,GM-2020-CrystallizationRandomMatrixOrbits,AGVP-2023-FiniteFreeCumulants}.\footnote{As some of these references discuss, the central definitions are already present, albeit in very different context, in much earlier literature on the geometry of polynomials \cite{Walsh-1922-LocationRootsPolynomials,Szego-1922-GraceWurzeln}.}

The key definitions of this theory are as follows.
\begin{definition}[Finite free convolutions]
    For complex univariate polynomials
    \begin{align*}
        p(x) &= \sum_{i=0}^n x^{n-i} (-1)^i a_i, \\
        q(x) &= \sum_{i=0}^n x^{n-i} (-1)^i b_i
    \end{align*}
    of degree at most $n$ with $a_i, b_i \in \bC$, the \emph{$n$th finite free additive convolution} of $p$ and $q$ is the polynomial
    \begin{align*}
        (p \boxplus_n q)(x) \colonequals \sum_{k=0}^n x^{n-k} (-1)^k \sum_{\substack{0 \leq i, j \leq k \\ i+j=k}} \frac{(n-i)!(n-j)!}{n!(n-k)!}a_ib_j,
    \end{align*}
    and the \emph{$n$th finite free multiplicative convolution} of $p$ and $q$ is the polynomial
    \begin{align*}
        (p \boxtimes_n q)(x) \colonequals \sum_{i=0}^n x^{n-i} (-1)^i \frac{a_ib_i}{\binom{n}{i}}.
    \end{align*}
\end{definition}

These operations have many equivalent definitions.
Their connection to permanents is clarified by expressions in terms of the roots of the polynomials.
The additive identity below is \cite[Theorem~2.11]{MSS-2022-FiniteFreeConvolution}, and the multiplicative identity follows from \cite[Lemma~2.6 and Theorem~2.12]{MSS-2022-FiniteFreeConvolution} by taking the two matrices involved there to be diagonal.
\begin{proposition}
    \label{prop:real-rooted-ffc}
    For monic real-rooted polynomials
    \begin{align*}
        p(x) &= \prod_{i=1}^n (x - \alpha_i), \\
        q(x) &= \prod_{i=1}^n (x - \beta_i)
    \end{align*}
    of degree $n$, their finite free additive convolution and finite free multiplicative convolution take the forms
    \begin{align*}
        (p \boxplus_n q)(x) &= \frac{1}{n!} \sum_{\pi \in \calS_n} \prod_{i=1}^n (x - \alpha_i - \beta_{\pi(i)}),\\
        (p \boxtimes_n q)(x) &= \frac{1}{n!} \sum_{\pi \in \calS_n} \prod_{i=1}^n (x - \alpha_i \beta_{\pi(i)}).
    \end{align*}
\end{proposition}

Our second main result is the following monotonicity inequality for the finite free additive and multiplicative convolutions of real-rooted polynomials.
We give the short proof deriving these from Theorem~\ref{thm:perm-inequality} immediately.
Below we write $\oone_n \in \bR^n$ for the vector all of whose entries equal 1 (so that $\JJ_n = \oone_n\oone_n^{\top}$).
\begin{theorem}\label{thm:ffc-inequality}
    Let $p(x), q(x)$ be monic real-rooted polynomials of degree $n$. Then, for every $x \in \bR$,
    \begin{align*}
        (p \boxplus_n q)^2(x) &\ge (p^2 \boxplus_{2n} q^2)(x),\\
        (p \boxtimes_n q)^2(x) &\ge (p^2 \boxtimes_{2n} q^2)(x).
    \end{align*}
\end{theorem}

\begin{proof}
    Let $p(x) = \prod_{i=1}^n (x - \alpha_i)$ and $q(x) = \prod_{i=1}^n (x - \beta_i)$ be two monic real-rooted polynomials of degree $n$. Let $\aalpha \colonequals (\alpha_1, \dots, \alpha_n)^\top, \bbeta \colonequals (\beta_1, \dots, \beta_n)^\top \in \bR^n$ be the vectors of roots of $p$ and $q$.

    First we prove the additive convolution inequality. By Proposition~\ref{prop:real-rooted-ffc},
    \begin{align*}
        (p \boxplus_n q)(x) &= \frac{1}{n!} \sum_{\pi \in \calS_n} \prod_{i=1}^n (x - \alpha_i - \beta_{\pi(i)})\\
        &= \frac{1}{n!} \per\left(
            x\JJ_n - \aalpha \oone_n^\top - \oone_n \bbeta^\top
        \right).
    \end{align*}
    Similarly, since $p^2(x) = \prod_{i=1}^n (x - \alpha_i)^2$ and $q^2 = \prod_{i=1}^n (x - \beta_i)^2$ are monic real-rooted polynomials of degree $2n$ with root vectors $\begin{pmatrix}
        \aalpha\\
        \aalpha
    \end{pmatrix}, \begin{pmatrix}
        \bbeta\\
        \bbeta
    \end{pmatrix} \in \bR^{2n}$, we have
    \begin{align*}
        (p^2 \boxplus_{2n} q^2)(x) &= \frac{1}{(2n)!} \per\left(
            x\JJ_{2n} - \begin{pmatrix}
        \aalpha\\
        \aalpha
    \end{pmatrix} \oone_{2n}^\top - \oone_{2n} \begin{pmatrix}
        \bbeta\\
        \bbeta
    \end{pmatrix}^\top
        \right)\\
        &= \frac{1}{(2n)!}\per\left(\left(x\JJ_n - \aalpha \oone_n^\top - \oone_n \bbeta^\top\right) \otimes \JJ_2\right).
    \end{align*}
    Observe that $x\JJ_n - \aalpha \oone_n^\top - \oone_n \bbeta^\top = (x \oone_n - \aalpha) \oone_n^\top - \oone_n \bbeta^\top$ has rank at most $2$. By Theorem~\ref{thm:perm-inequality}, we conclude that
    \begin{align*}
        (p \boxplus_n q)^2(x) &= \frac{1}{(n!)^2} \per^2\left(
            x\JJ_n - \aalpha \oone_n^\top - \oone_n \bbeta^\top
        \right)\\
        &\ge \frac{1}{(2n)!}\per\left(\left(x\JJ_n - \aalpha \oone_n^\top - \oone_n \bbeta^\top\right) \otimes \JJ_2\right)\\
        &= (p^2 \boxplus_{2n} q^2)(x).
    \end{align*}

    Next we prove the multiplicative convolution inequality. By Proposition~\ref{prop:real-rooted-ffc},\begin{align*}
        (p \boxtimes_n q)(x) &= \frac{1}{n!} \sum_{\pi \in \calS_n} \prod_{i=1}^n (x - \alpha_i \beta_{\pi(i)})\\
        &= \frac{1}{n!} \per\left(
            x\JJ_n - \aalpha \bbeta^\top
        \right).
    \end{align*}
    Similarly, since $p^2(x) = \prod_{i=1}^n (x - \alpha_i)^2$ and $q^2(x) = \prod_{i=1}^n (x - \beta_i)^2$ are monic real-rooted polynomials of degree $2n$ with root vectors $\begin{pmatrix}
        \aalpha\\
        \aalpha
    \end{pmatrix}$ and $\begin{pmatrix}
        \bbeta\\
        \bbeta
    \end{pmatrix}$, we have
    \begin{align*}
        (p^2 \boxtimes_{2n} q^2)(x) &= \frac{1}{(2n)!} \per\left(
            x\JJ_{2n} - \begin{pmatrix}
        \aalpha\\
        \aalpha
    \end{pmatrix}  \begin{pmatrix}
        \bbeta\\
        \bbeta
    \end{pmatrix}^\top
        \right)\\
        &= \frac{1}{(2n)!}\per\left(\left(x\JJ_n - \aalpha \bbeta^\top\right) \otimes \JJ_2\right).
    \end{align*}
    Observe that $x\JJ_n - \aalpha \bbeta^\top $ has rank at most $2$. By Theorem~\ref{thm:perm-inequality}, we conclude that
    \begin{align*}
        (p \boxtimes_n q)^2(x) &= \frac{1}{(n!)^2} \per^2\left(
            x\JJ_n - \aalpha \bbeta^\top
        \right)\\
        &\ge \frac{1}{(2n)!}\per\left(\left(x\JJ_n - \aalpha  \bbeta^\top\right) \otimes \JJ_2\right)\\
        &= (p^2 \boxtimes_{2n} q^2)(x). \qedhere
    \end{align*}
\end{proof}

\begin{remark}
As we state precisely in Corollary~\ref{cor:ffc-inequality-SoS},
  this proof and the sum-of-squares certificate obtained in 
  Theorem~\ref{thm:perm-inequality-SoS} provide 
  sum-of-squares certificates of the inequalities above 
  in terms of the $x$ and the roots of the polynomials $p$ and $q$.
\end{remark}

One corollary of this relevant to the relationship between finite and classical free probability is as follows.
Write $\maxroot(p)$ for the largest root of a real-rooted polynomial $p$ and $\mu_p$ for the empirical distribution of roots of $p$.
For compactly supported probability measures $\mu$ and $\nu$, write $\mu \boxplus \nu$ for their classical free additive convolution \cite{Voiculescu-1986-AdditionNoncommuting} and $\mu \boxtimes \nu$ for their classical free multiplicative convolution, the latter provided one of them is supported on $[0, \infty)$ \cite{Voiculescu-1987-MultiplicationNoncommuting}.
Finally, for a compactly supported probability measure $\mu$, write $\maxsupp(\mu)$ for the largest value in the support of $\mu$.

We first recall the following relationships between the finite and classical free convolutions.
The first has been stated in and follows directly from prior work.
\begin{proposition}[\cite{MSS-2022-FiniteFreeConvolution}, Theorem~4.2 of \cite{Srivastava-2024-RamanujanSurvey}]
    \label{prop:additive-finite-to-free-endpoint}
    Let $p$ and $q$ be monic real-rooted polynomials of degree $n$.
    Then,
    \begin{align*}
        \maxroot(p \boxplus_n q) \leq \maxsupp(\mu_p \boxplus \mu_q).
    \end{align*}
\end{proposition}
\noindent
The second, to the best of our knowledge, has not been stated directly before, but follows easily from existing results, essentially using a trick to reduce to Proposition~\ref{prop:additive-finite-to-free-endpoint}.
\begin{proposition}
    \label{prop:multiplicative-finite-to-free-endpoint}
    Let $p$ and $q$ be monic real-rooted polynomials of degree $n$, and suppose that at least one of them has only non-negative roots.
    Then, $p \boxtimes_n q$ is real-rooted and
    \begin{align*}
        \maxroot(p \boxtimes_n q) \leq \maxsupp(\mu_p \boxtimes \mu_q).
    \end{align*}
\end{proposition}
\noindent
We give the proof in Section~\ref{sec:mult-endpoint}.

Our corollary of Theorem~\ref{thm:ffc-inequality} further shows that sequences of finite free convolutions of powers of polynomials interpolate between the endpoints of these inequalities.
\begin{corollary}
    Let $p$ and $q$ be monic real-rooted polynomials of degree $n$.
    Then, the following hold:
    \begin{enumerate}
    \item The sequence $\maxroot(p^{2^k} \boxplus_{2^k n} q^{2^k})$ over $k \geq 0$ is monotonically non-decreasing and converges to $\maxsupp(\mu_p \boxplus \mu_q)$.
    \item If either $p$ or $q$ has only non-negative roots, then the sequence $\maxroot(p^{2^k} \boxtimes_{2^k n} q^{2^k})$ over $k \geq 0$ is monotonically non-decreasing and converges to $\maxsupp(\mu_p \boxtimes \mu_q)$.
    \end{enumerate}
\end{corollary}
\noindent
For the second result, we note that two issues arise if the non-negativity condition is not included.
First, it can be that $p \boxtimes_n q$ is no longer real-rooted even if $p$ and $q$ are (for instance, we have $(x^2-1)\boxtimes_2(x^2-1)=x^2+1$).
Second, the standard self-adjoint multiplicative free convolution on
  the real line used here assumes that at least one input measure is
  supported on $[0,\infty)$.

\begin{proof}
We give the proof for the additive free convolution; the multiplicative case follows by an identical argument.
Define $\rho \colonequals \maxroot(p \boxplus_n q)$.
Then, $(p \boxplus_n q)(\rho) = 0$, and so Theorem~\ref{thm:ffc-inequality} implies that $(p^2 \boxplus_{2n} q^2)(\rho) \leq 0$.
But, $p^2 \boxplus_{2n} q^2$ is monic and thus positive on sufficiently large inputs, so we find
\[ \maxroot(p \boxplus_n q) \leq \maxroot(p^2 \boxplus_{2n} q^2). \]
Iterating this, we find the chain of inequalities
\[ \maxroot(p \boxplus_n q) \leq \maxroot(p^2 \boxplus_{2n} q^2) \leq \maxroot(p^4 \boxplus_{4n} q^4) \leq  \cdots \leq \maxroot(p^{2^k} \boxplus_{2^k n} q^{2^k}). \]
Corollary~5.5 of \cite{AP-2018-CumulantsFiniteFreeConvolution} gives the weak convergence
\[   \mu_{p^{\ell}\boxplus_{\ell n}q^{\ell}}
  \xrightarrow[\ell \to \infty]{\text{(w)}}
  \mu_p\boxplus\mu_q. \]
This weak convergence implies that
\[ \maxsupp(\mu_p \boxplus \mu_q) \leq \liminf_{\ell \to \infty} \left\{\maxroot(p^{\ell}\boxplus_{\ell n}q^{\ell})\right\}, \]
since every neighborhood of $\maxsupp(\mu_p \boxplus \mu_q)$ has positive measure under $\mu_p \boxplus \mu_q$.
On the other hand, Proposition~\ref{prop:additive-finite-to-free-endpoint} applied to $p^{\ell}$ and $q^{\ell}$ gives
\[ \maxroot(p^{\ell}\boxplus_{\ell n}q^{\ell}) \leq \maxsupp(\mu_p \boxplus \mu_q) \]
for every $\ell \geq 1$, since $\mu_{p^{\ell}} = \mu_p$ and $\mu_{q^{\ell}} = \mu_q$.
It follows that the maximum roots converge to $\maxsupp(\mu_p \boxplus \mu_q)$ as $\ell \to \infty$, and hence also along the subsequence $\ell=2^k$.

If either $p$ or $q$ has only non-negative roots, then the same argument applies to the free multiplicative convolutions.
Indeed, Proposition~2.16 of \cite{Fujie-2025-RegularityConvergenceFiniteFreeConvolutions} shows that all of the finite free multiplicative convolutions involved are real-rooted.
Theorem~1.4 of \cite{AGVP-2023-FiniteFreeCumulants} gives the weak convergence of their empirical root distributions to $\mu_p \boxtimes \mu_q$.
The result then follows by using Proposition~\ref{prop:multiplicative-finite-to-free-endpoint} instead of Proposition~\ref{prop:additive-finite-to-free-endpoint}.
\end{proof}

\subsection{Proof Techniques and Role of Artificial Intelligence}
\label{sec:proof-techniques}

We arrived at the proof of Theorem~\ref{thm:perm-inequality}, our main contribution, by a combination of ``classical'' computer assistance, direct analysis by the authors, and assistance from recent AI systems.

The sum-of-squares proofs we give, which are sequences of (rather complicated) polynomial equalities rather than inequalities, were found by solving appropriate semidefinite programs (SDPs) and identifying conjectural exact forms of the numerical sum-of-squares proofs they provided, which could be made highly symmetric with appropriate regularizations.
We note that, once such exact conjectural identities with rational coefficients are identified, they are far less expensive to test in larger dimensions than it would be to continue solving larger SDPs.
Performing such tests (in our case for degrees through $n = 49$), we reached a high level of confidence in these conjectures.

AI then provided several of the key insights in proving these identities; we present a streamlined and reorganized version of a proof first discovered in an interactive session with GPT-5.2~Thinking and Gemini~3.1~Pro in February~2026.
This was achieved by our prompting these systems to use tools from representation theory to explain the symmetry in the proposed sum-of-squares proof and attempt a symbolic proof.

In addition to many smaller differences, several parts of the proof were originally phrased (as requested) in the language of representation theory, while we present an elementary and considerably simplified proof here for the sake of clarity.
We summarize the roles representation theory originally played below, which may illuminate the roles of some of the objects involved.
See Definition~\ref{def:hom} below for the definition of the $\calW_m$ spaces.
\begin{itemize}
\item Corollary~\ref{cor:cg} is an elementary rephrasing of the Clebsch-Gordan decomposition of $\calW_m^{\otimes 2}$, while Corollary~\ref{cor:Bm-even-decomp} is a consequence of the restricted Clebsch-Gordan decomposition of $\Sym^2(\calW_n)$.
After complexification, both decompositions are into irreducible representations of $\SU_2$, each of which has multiplicity $1$.
\item The transvectant $\tau_{2k}$ shows up in Theorem~\ref{thm:inner-product-decom} as it is the unique $\SU_2$-intertwiner (up to rescaling, and again after complexification) from $\Sym^2(\calW_n)$ to $\calW_{2n-4k}$.
\item Proposition~\ref{prop:transvectant-identity} was first established by realizing the right-hand side of the identity with an $\SU_2$-intertwiner, and using the uniqueness of the transvectant as the intertwiner from $\Sym^2(\calW_n)$ to $\calW_{2n-4k}$.
\end{itemize}

\subsection{Acknowledgements}

The authors would like to thank Nikhil Srivastava
and Jorge Garza-Vargas for a helpful discussion on the project.

\section{Preliminaries on Homogeneous Polynomials}

Our arguments will relate permanents and finite free convolutions to certain inner products of bivariate homogeneous polynomials.
We prepare by establishing some general theory around inner product spaces of homogeneous polynomials.

\begin{definition}[Homogeneous polynomials]
    \label{def:hom}
    We write $\mathbb{R}[x_1, \dots, x_d]^{\hom}_n$ for the space of real homogeneous polynomials in the variables $(x_1, \dots, x_d)$ of degree $n$.

    We will also use some special notation for the case $d = 2$.
    In this case, we always call the variables involved $x$ and $y$, so that such a polynomial is of the form $\sum_{k=0}^n a_k x^k y^{n-k}$, where $a_k \in \bR$ for $0 \le k \le n$.
    We use $\calW_n \colonequals \bR[x, y]_n^{\hom}$ to denote the space of these real bivariate homogeneous polynomials of degree $n$.
\end{definition}

\subsection{Apolar Inner Product Structure}

We will use the following inner product of homogeneous polynomials,  also sometimes called the \emph{Fischer} or \emph{Bombieri inner product}.
Up to normalization, this is the same as the definition in Equation~(2.3) of \cite{Reznick-1996-HomogeneousPolynomial}.
\begin{definition}[Apolar inner product]
    We define an inner product on $\bR[x_1, \dots, x_d]_n^{\hom}$ for each $d \geq 1$ and $n \geq 0$ as follows.
    First, on monomials, we set
    \[ \langle x_1^{a_1} \cdots x_d^{a_d}, x_1^{b_1} \cdots x_d^{b_d} \rangle \colonequals \left\{\begin{array}{ll} a_1! \cdots a_d! & \text{if } a_1 = b_1, \dots, a_d = b_d, \\ 0 & \text{otherwise} \end{array}\right\}. \]
    We then extend to arbitrary inner products $\langle f, g \rangle$ for $f, g \in \bR[x_1, \dots, x_d]_n^{\hom}$ bilinearly.
\end{definition}

The apolar inner product has a useful reinterpretation in terms of differentiation, with respect to which it has convenient algebraic properties as follows.
The following proposition comes from \cite{Reznick-1996-HomogeneousPolynomial}, where parts 1, 2, and 3 are in Theorem~2.11 and part 4 is Equation~(2.7).

\begin{proposition}[Differential interpretation of the apolar inner product]
    \label{prop:apolar-diff}
    For $f \in \mathbb{R}[x_1, \dots, x_d]$, write $f(\bm \partial) = f(\partial_{x_1}, \dots, \partial_{x_d})$ for the corresponding differential operator (with constant coefficients).
    The following statements hold:
    \begin{enumerate}
        \item If $f,g \in \bR[x_1, \dots, x_d]_n^{\hom}$, then
        \[ \langle f,g\rangle
            = f(\bm\partial)g \\
            = g(\bm\partial)f. \]
        In particular, both differential expressions above are constants.

        \item If $f \in \bR[x_1, \dots, x_d]_m^{\hom}$, $g \in \bR[x_1, \dots, x_d]_n^{\hom}$, and
        $h \in \bR[x_1, \dots, x_d]_{m + n}^{\hom}$, then
        \begin{align*}
            \langle fg,h\rangle
            =
            \left\langle
                g, f(\bm\partial)h
            \right\rangle.
        \end{align*}
        Thus, with respect to the apolar inner products on the relevant homogeneous components, the adjoint of multiplication by $f$ is the differential operator $f(\bm\partial)$.

        \item In particular, if $g \in \bR[x_1, \dots, x_d]_n^{\hom}$ and
        $h \in \bR[x_1, \dots, x_d]_{n + 1}^{\hom}$, then, for all $i \in [d]$,
        \[ \langle x_i g,h\rangle
            = \langle g,\partial_{x_i} h\rangle \]

        \item If $f \in \bR[x_1, \dots, x_d]_n^{\hom}$ and $(u_1, \dots, u_d)\in\mathbb{R}^d$, then
        \begin{align*}
            f(u_1, \dots, u_d) = \frac{1}{n!} \left\langle (u_1x_1 + \cdots + u_d x_d)^n,f(x_1, \dots, x_d)\right\rangle.
        \end{align*}
    \end{enumerate}
\end{proposition}

\subsection{Ideal-Harmonic Decomposition}

Proposition~\ref{prop:apolar-diff} has the following elegant consequence.
This follows, for instance, from the combination of Theorems~2.18 and~4.7 of \cite{Reznick-1996-HomogeneousPolynomial}.
\begin{proposition}[Ideal-harmonic decomposition]
    \label{prop:ideal-harmonic}
    Let $1 \leq m \leq n$ and $f \in \bR[x_1, \dots, x_d]_m^{\hom} \setminus \{0\}$.
    Define the subspaces of $\bR[x_1, \dots, x_d]_n^{\hom}$
    \begin{align*}
    \mathcal{I}_n = \mathcal{I}_{f,n} &\colonequals \{fg: g \in \bR[x_1, \dots, x_d]_{n - m}^{\hom}\}, \\
    \mathcal{H}_n = \mathcal{H}_{f,n} &\colonequals \{g \in \bR[x_1, \dots, x_d]_n^{\hom}: f(\bm \partial)g = 0\},
    \end{align*}
    respectively the homogeneous \emph{ideal} generated by $f$ and the \emph{$f$-harmonic} polynomials.
    Then, we have
    \[ \bR[x_1, \dots, x_d]_n^{\hom} = \mathcal{H}_{n} \oplus \mathcal{I}_{n}, \]
    and $\mathcal{H}_{n}$ and $\mathcal{I}_{n}$ are one another's orthogonal complements in $\bR[x_1, \dots, x_d]_n^{\hom}$ under the apolar inner product.
    Further, if $k$ is the greatest integer such that $k m \leq n$, then
    \[ \bR[x_1, \dots, x_d]_n^{\hom} = \mathcal{H}_n \oplus f\mathcal{H}_{n - m} \oplus \cdots \oplus f^k \mathcal{H}_{n - km}. \]
\end{proposition}

\subsection{\texorpdfstring{$2 \times 2$}{2-by-2} Matrices and Clebsch-Gordan Decomposition}

We will in particular apply the above theory to homogeneous polynomials in four variables, belonging to the spaces $\bR[x_1, x_2, y_1, y_2]_n^{\hom}$, which we view as being functions of a $2 \times 2$ matrix
\[ \left[\begin{array}{cc} x_1 & y_1 \\ x_2 & y_2 \end{array}\right]. \]
We consider the ideal-harmonic decomposition with respect to the \emph{determinant} polynomial,
\[ D = D(x_1, x_2, y_1, y_2) = x_1 y_2 - y_1 x_2. \]
We denote the associated differential operator, sometimes called \emph{Cayley's operator} or \emph{process}, by
\[ \Omega = D(\partial_{x_1}, \partial_{x_2}, \partial_{y_1}, \partial_{y_2}) = \partial_{x_1}\partial_{y_2} - \partial_{y_1}\partial_{x_2}. \]
Note that, for any $n \geq 2$, this is a linear operator $\Omega: \bR[x_1, x_2, y_1, y_2]_n^{\hom} \to \bR[x_1, x_2, y_1, y_2]_{n - 2}^{\hom}$, and on $n < 2$ it acts on $\bR[x_1, x_2, y_1, y_2]_n^{\hom}$ as the zero operator.

We further denote by $\calB_m$ the subspace of $\bR[x_1, x_2, y_1, y_2]_{2m}^{\hom}$ consisting of those polynomials that are separately homogeneous of degree $m$ in the pair $(x_1, y_1)$ and homogeneous of degree $m$ in the pair $(x_2, y_2)$.
Recalling that we denote by $\calW_m$ the homogeneous bivariate polynomials of degree $m$, we may also identify $\calB_m = \calW_m \otimes \calW_m$.

We then have $\Omega: \calB_m \to \calB_{m - 1}$ and, identifying $D$ with the operation of multiplication by $D$, $D: \calB_{m - 1} \to \calB_m$.
We define the kernels of each degree among these balanced polynomials,
\[ \calK_m \colonequals \ker(\Omega) \cap \calB_m = \calH_{D, 2m} \cap \calB_m. \]

Restricting Proposition~\ref{prop:ideal-harmonic} to $\calB_m \subset \bR[x_1, x_2, y_1, y_2]_{2m}^{\hom}$ then gives the following.
\begin{corollary}
    \label{cor:cg}
    For each $m \geq 1$,
    \[ \calB_m = \calK_m \oplus D\calK_{m - 1} \oplus \cdots \oplus D^m \calK_0, \]
    and all summand subspaces are mutually orthogonal under the apolar inner product.
\end{corollary}
\begin{proof}
The direct sum decomposition follows by iterating Proposition~\ref{prop:ideal-harmonic}.
For mutual orthogonality, fix $0 \leq j < \ell \leq m$, $K \in \calK_{m-j}$, and $L \in \calK_{m-\ell}$.
Repeatedly using the adjointness of $D$ and $\Omega$ and Proposition~\ref{prop:Omega-relations} gives
\[
    \langle D^j K, D^\ell L \rangle
    = \langle K, \Omega^j D^\ell L \rangle
    = c \langle K, D^{\ell-j}L \rangle
\]
for some scalar $c$.
Since $\ell-j \geq 1$,
\[
    \langle K, D^{\ell-j}L \rangle
    = \langle \Omega K, D^{\ell-j-1}L \rangle
    = 0,
\]
which proves the claim.
\end{proof}
\noindent
This decomposition, in this language just a special case of Proposition~\ref{prop:ideal-harmonic}, is essentially a rephrasing of the classical Clebsch-Gordan decomposition of representation theory and mathematical physics.
Indeed, after complexification, $\calW_m \cong \Sym^m(\bC^2)$ is the irreducible $\SU_2$-representation of highest weight $m$, and the Clebsch-Gordan formula gives $\calB_m \otimes_{\bR} \bC \cong
  \bigoplus_{j=0}^m \Sym^{2m-2j}(\bC^2)$ \cite[Section~11.2]{FH-2004-RepresentationTheory}.
Under this identification, $D$ and $\Omega$ are respectively invariant and equivariant for the diagonal $\SU_2$-action, and $(D^j\calK_{m-j}) \otimes_{\bR} \bC$ is precisely the summand $\Sym^{2m-2j}(\bC^2)$.

Let us also register how this decomposition interacts with another operator on $\calB_m$: denote by $\Sigma: \calB_m \to \calB_m$ the operation swapping the two variable pairs,
\[ (\Sigma F)(x_1, y_1, x_2, y_2) \colonequals F(x_2, y_2, x_1, y_1). \]
We have
\[ \Sigma D = -D. \]
The following says that this is the only non-trivial action of $\Sigma$ in the decomposition of Corollary~\ref{cor:cg}; its action fixes all of the factors lying in the kernel spaces.
\begin{proposition}
    \label{prop:K-expansion}
    Suppose $K(x_1, y_1, x_2, y_2) \in \calK_m$ with coefficient expansion
    \begin{equation}
    K(x_1, y_1, x_2, y_2) = \sum_{i, j = 0}^m k_{ij} x_1^i y_1^{m - i} x_2^j y_2^{m - j}. \label{eq:K-expansion}
    \end{equation}
    Then, there are scalars $c_k$ for $k = 0, 1, \dots, 2m$ such that
    \[ k_{ij} = c_{i + j} \binom{m}{i} \binom{m}{j}. \]
    Therefore, we have $\Sigma K = K$.
\end{proposition}
\begin{proof}
Since $\Omega K = 0$, comparing the coefficient of $x_1^i y_1^{m-1-i}x_2^j y_2^{m-1-j}$ for $0 \leq i,j \leq m-1$ on either side of this equation gives
\[ (i+1)(m-j)k_{i+1,j} = (m-i)(j+1)k_{i,j+1}. \]
Set $c_{ij} \colonequals k_{ij} / (\binom{m}{i}\binom{m}{j})$.
Using $(i+1)\binom{m}{i+1} = (m-i)\binom{m}{i}$ and $(m-j)\binom{m}{j} = (j+1)\binom{m}{j+1}$, the preceding relation becomes $c_{i+1,j} = c_{i,j+1}$.
Thus $c_{ij}$ is constant on each set where $i+j=k$, and denoting this constant by $c_k$ gives the claimed formula.
In particular, $k_{ij}=k_{ji}$, so swapping the two variable pairs leaves $K$ unchanged.
\end{proof}

\begin{corollary}
    \label{cor:Bm-even-decomp}
    Suppose that $F \in \calB_m$ satisfies $\Sigma F = F$.
    Then, there exist $K_k \in \calK_{m - 2k}$ for each $k = 0, 1, \dots, \lfloor m / 2 \rfloor$ such that
    \begin{equation}
    F = \sum_{k = 0}^{\lfloor m / 2 \rfloor} D^{2k} K_k. \label{eq:even-expansion}
    \end{equation}
\end{corollary}
\begin{proof}
    This follows immediately from using Corollary~\ref{cor:cg} on $F$, that $F = \frac{1}{2} F + \frac{1}{2}\Sigma F$, and that $\Sigma D = -D$.
\end{proof}
\noindent
As we mentioned in Section~\ref{sec:proof-techniques}, while Corollary~\ref{cor:cg} is the Clebsch-Gordan decomposition of all of $\calB_m = \calW_m \otimes \calW_m$, the above is the analogous decomposition of $\Sym^2(\calW_m)$, which may be identified with the subspace of $\calB_m$ fixed by $\Sigma$.

We note that one simple class of $F$ having this invariance is any factorized $F(x_1, y_1, x_2, y_2) = f(x_1, y_1)f(x_2, y_2)$ for some $f \in \calW_{m}$.
We will use this construction to apply Corollary~\ref{cor:Bm-even-decomp} below.

\subsection{Diagonal Restriction, Transvectants, and Bivariate Inner Products}

Next, we derive consequences of the above results on polynomials in four variables for polynomials in two variables.
Recall that we denote by $\calW_n$ the space of real homogeneous polynomials of degree $n$ in two variables $x, y$.
The bridge between the above results and polynomials on $\calW_n$ is the map setting $(x_1, y_1) = (x_2, y_2)$: we define $\calR : \calB_m \to \calW_{2m}$ by, for $F \in \calB_m$, setting
\[ (\calR F)(x, y) \colonequals F(x, y, x, y) \in \calW_{2m}. \]
This operation is sometimes called the \emph{restitution operator} or the \emph{trace operator} in the invariant theory literature.
The following are simple calculations using what we have shown already.
\begin{proposition}
    \label{prop:diag-inner-product}
    Suppose that $K(x_1, y_1, x_2, y_2) \in \calK_m$ has the expansion in \eqref{eq:K-expansion}.
    Then, we have
    \[ (\calR K)(x, y) = \sum_{k = 0}^{2m} c_k \binom{2m}{k} x^k y^{2m - k}. \]
\end{proposition}
\begin{proof}
    The result follows from grouping like terms in the result of Proposition~\ref{prop:K-expansion} and applying Vandermonde's identity for binomial coefficients.
\end{proof}

\begin{proposition}
    \label{prop:apolar-diag}
    Suppose that $K, L \in \calK_m$.
    Then,
    \[ \binom{2m}{m} \langle K, L \rangle = \langle \calR K, \calR  L \rangle, \]
    with apolar inner products on either side of the equation.
\end{proposition}
\begin{proof}
    Let $K$ and $L$ admit expansions as in \eqref{eq:K-expansion} with coefficients $k_{ij}$ and $\ell_{ij}$, respectively.
    By Proposition~\ref{prop:K-expansion}, there exist $c_k$ and $d_k$ such that $k_{ij} = c_{i + j} \binom{m}{i} \binom{m}{j}$ and $\ell_{ij} = d_{i + j} \binom{m}{i} \binom{m}{j}$.
    Working from the definition of apolar inner product and then using Vandermonde's identity and Proposition~\ref{prop:diag-inner-product}, we find
    \begin{align*}
        \langle K, L \rangle
        &= \sum_{i, j = 0}^m k_{ij} \ell_{ij} \cdot i!(m - i)! j!(m - j)! \\
        &= (m!)^2 \sum_{i, j = 0}^m c_{i + j} d_{i + j} \binom{m}{i} \binom{m}{j} \\
        &= (m!)^2 \sum_{k = 0}^{2m} c_k d_k \binom{2m}{k} \\
        &= \frac{(m!)^2}{(2m)!} \langle \calR K, \calR L \rangle,
    \end{align*}
    as claimed.
\end{proof}

\begin{definition}
    \label{def:tv}
    For $f, g \in \calW_n$ and $0 \leq k \leq n$, 
    the \emph{$k$th transvectant} of $f$ and $g$ is defined as
    \[ \tau_k(f, g) \colonequals \calR \Omega^k P \in \calW_{2n - 2k},  \]
where $P(x_1, y_1, x_2, y_2) \colonequals f(x_1, y_1)g(x_2, y_2).$
\end{definition}

One use and interpretation of the transvectants is that, for $F(x_1, y_1, x_2, y_2) = f(x_1, y_1) f(x_2, y_2)$ as mentioned above, they extract the various components of the expansion \eqref{eq:even-expansion}, subject to the $\calR$ restriction.
We first establish some useful properties of $D$ and $\Omega$:

\begin{proposition}
    \label{prop:Omega-relations}
    Define the operators
    \begin{align*}
        E_1 &\colonequals x_1 \partial_{x_1} + y_1 \partial_{y_1}, \quad \text{and}\\
        E_2 &\colonequals x_2 \partial_{x_2} + y_2 \partial_{y_2}.
    \end{align*}
    Then, the following hold:
    \begin{enumerate}
        \item $\Omega D - D\Omega = E_1 + E_2 + 2$.
        \item If $K \in \calK_k$ and $\ell \geq 1$, then $\Omega D^{\ell}K = \ell(2k + \ell + 1)D^{\ell - 1}K$.
        \item If $K \in \calK_k$ and $\ell \geq 1$, then $\Omega^{\ell}(D^{\ell} K) = \ell! \frac{(2k+\ell+1)!}{(2k+1)!} K$.
    \end{enumerate}
\end{proposition}
\begin{proof}
    The first result follows by the product rule.
    This result may also be rewritten in the form
    \[ \Omega D = D\Omega + (E_1 + E_2 + 2), \]
    which will be useful below.

    For the second result, since $K \in \calK_k \subseteq \calB_k$, $K$ has degree $k$ in both pairs $(x_1, y_1)$ and $(x_2, y_2)$, and so $D^{\ell - 1}K$ has degree $k + \ell - 1$ in both of these pairs.
    Thus, $(E_1 + E_2 + 2)D^{\ell - 1}K = (2k + 2\ell)D^{\ell - 1}K$.
    We have $\Omega K = 0$ and the recursive relation
    \begin{align*}
    \Omega D^{\ell} K
    &= D\Omega D^{\ell - 1} K + (E_1 + E_2 + 2)D^{\ell - 1} K \\
    &= D\Omega D^{\ell - 1} K + (2k + 2\ell)D^{\ell - 1} K,
    \end{align*}
    and by induction we find $\Omega D^{\ell} K = \ell(2k + \ell + 1) D^{\ell - 1} K$, as claimed.
    Iteratively applying this then gives the third result.
\end{proof}

\begin{proposition}
    \label{prop:diag-tv}
    Suppose that $f \in \calW_n$, and define $F(x_1, y_1, x_2, y_2) \colonequals f(x_1, y_1) f(x_2, y_2) \in \calB_n$.
    Let the expansion of $F$ of form \eqref{eq:even-expansion} have terms
    $K_k \in \calK_{n - 2k}$ for $k = 0, 1, \dots, \lfloor n / 2 \rfloor$.
    Then,
    \[ \calR K_k = \frac{(2n-4k+1)!}{(2k)!(2n - 2k + 1)!} \tau_{2k}(f, f). \]
\end{proposition}
\begin{proof}
    We have by definition of the transvectant and linearity of $\Omega$ and $\calR$ that
    \begin{align*}
        \tau_{2k}(f, f)
        &= \calR \Omega^{2k} F \\
        &= \sum_{\ell = 0}^{\lfloor n / 2 \rfloor} \calR  \Omega^{2k} D^{2\ell} K_{\ell}.
    \end{align*}
    We claim that only the term $k = \ell$ here is non-zero.
    Indeed, by Proposition~\ref{prop:Omega-relations}, if $k > \ell$ then this is some multiple of a power of $\Omega$ applied to $K_{\ell}$, which gives zero, while if $k < \ell$ then it is a multiple of $D$, which gives zero after applying $\calR$.
    Applying part 3 of Proposition~\ref{prop:Omega-relations} then gives
    \begin{align*}
        \tau_{2k}(f, f)
        &= \calR \Omega^{2k}D^{2k} K_{k} \\
        &= \calR \left(\frac{(2k)!(2n-2k+1)!}{(2n-4k+1)!} K_k\right),
    \end{align*}
    and the result follows by linearity of $\calR$.
\end{proof}

\begin{theorem}\label{thm:inner-product-decom}
    For any $f, g \in \calW_n$, we have
    \[ \langle f, g \rangle^2 = \sum_{k = 0}^{\lfloor n/2 \rfloor} c_{n, k} \cdot \langle \tau_{2k}(f, f), \tau_{2k}(g, g) \rangle, \]
    with coefficients
    \[ c_{n, k} = \frac{(2n-4k+1)!((n-2k)!)^2}{(2k)!(2n-2k+1)!(2n-4k)!}. \]
\end{theorem}
\begin{proof}
    Define $F, G \in \calB_n$ by
    \begin{align*}
        F(x_1, y_1, x_2, y_2) &\colonequals f(x_1, y_1)f(x_2, y_2), \\
        G(x_1, y_1, x_2, y_2) &\colonequals g(x_1, y_1) g(x_2, y_2).
    \end{align*}
    A direct calculation shows that, on the one hand,
    \[ \langle f, g \rangle^2 = \langle F, G \rangle. \]

    On the other hand, Corollary~\ref{cor:Bm-even-decomp} applies to both $F$ and $G$, and thus there exist $K_k, L_k \in \calK_{n - 2k}$ for each $k = 0, 1, \dots, \lfloor n/2 \rfloor$ such that
    \begin{align*}
        F &= \sum_{k = 0}^{\lfloor n / 2 \rfloor} D^{2k} K_k \\
        G &= \sum_{k = 0}^{\lfloor n / 2 \rfloor} D^{2k} L_k.
    \end{align*}
    By the orthogonality property of Corollary~\ref{cor:cg}, summands for different $k$ in these two sums are orthogonal under the apolar inner product.
    Thus, we have
    \begin{align*}
    \langle F, G \rangle
    &= \sum_{k = 0}^{\lfloor n / 2 \rfloor} \langle D^{2k} K_k, D^{2k} L_k \rangle
    \intertext{By the adjointness property from Proposition~\ref{prop:apolar-diff} followed by Proposition~\ref{prop:Omega-relations}, we have}
    &= \sum_{k = 0}^{\lfloor n / 2 \rfloor} \langle K_k, \Omega^{2k} D^{2k} L_k \rangle \\
    &= \sum_{k = 0}^{\lfloor n / 2 \rfloor} \frac{(2k)!(2n-2k+1)!}{(2n-4k+1)!} \cdot \langle K_k, L_k \rangle
    \intertext{Now, using Proposition~\ref{prop:apolar-diag} in each term,}
    &= \sum_{k = 0}^{\lfloor n / 2 \rfloor} \frac{(2k)!(2n-2k+1)!}{(2n-4k+1)!} \cdot \frac{((n - 2k)!)^2}{(2n - 4k)!} \cdot \langle \calR  K_k, \calR L_k \rangle
    \intertext{and finally, using Proposition~\ref{prop:diag-tv} in each remaining polynomial,}
    &= \sum_{k = 0}^{\lfloor n / 2 \rfloor} \frac{(2n-4k+1)!}{(2k)!(2n - 2k + 1)!} \cdot \frac{((n - 2k)!)^2}{(2n - 4k)!} \cdot \langle \tau_{2k}(f, f), \tau_{2k}(g, g) \rangle,
    \end{align*}
    giving the result.
\end{proof}

\section{Main Permanent Inequality: Proof of Theorem~\ref{thm:perm-inequality}} \label{sec:perm-SoS}

For a finite set $S$, we write $\mathcal{P}_2(S)$ for the set of perfect matchings of $S$ (or, equivalently, partitions into parts all of size exactly 2).
For ease of presentation, we first introduce the following polynomial.

\begin{definition}[$\Delta$-polynomial]
    Let $\UU = (\uu_1, \dots, \uu_n)^\top \in \bR^{n \times 2}$ be a matrix with row vectors $\uu_1, \dots, \uu_n \in \bR^2$. For any $S \subseteq [n]$ of even size, define
    \begin{align*}
        \Delta_S(\UU) \colonequals \sum_{M \in \calP_2(S)} \prod_{ \{i,j\} \in M } \det\nolimits^{2}\begin{pmatrix}
            \uu_i^\top \\
            \uu_j^\top
        \end{pmatrix}.
    \end{align*}
    By convention we take $\Delta_{\varnothing} = 1$.
    Note that $\Delta_S(\UU)$ is a homogeneous sum-of-squares polynomial in $\UU$ of degree $2|S|$.
\end{definition}
\noindent
The following is a straightforward fact for $\Delta$-polynomials.
\begin{lemma}\label{lem:delta-contraction}
    Let $\UU = (\uu_1, \dots, \uu_n)^\top \in \bR^{n \times 2}$, $S \subseteq [n]$ be of even size $|S| = s$, and $0 \le t \le s$ be even. Then,
    \begin{align*}
        \sum_{\substack{T \subseteq S:\\ |T| = t}} \Delta_T(\UU) \Delta_{S \setminus T}(\UU) &= \binom{\frac{s}{2}}{\frac{t}{2}}\Delta_S(\UU).
    \end{align*}
\end{lemma}

\begin{proof}
    Calculating directly,
    \begin{align*}
        &\sum_{\substack{T \subseteq S:\\ |T| = t}} \Delta_T(\UU) \Delta_{S \setminus T}(\UU) \\
        &= \sum_{\substack{T \subseteq S:\\ |T| = t}} \left(\sum_{M_1 \in \calP_2(T)} \prod_{ \{i,j\} \in M_1 } \det\nolimits^{2}\begin{pmatrix}
            \uu_i^\top \\
            \uu_j^\top
        \end{pmatrix}\right)\left( \sum_{M_2 \in \calP_2(S \setminus T)} \prod_{ \{i,j\} \in M_2 } \det\nolimits^{2}\begin{pmatrix}
            \uu_i^\top \\
            \uu_j^\top
        \end{pmatrix}\right).
    \end{align*}
    For every fixed choice of $T \subseteq S, M_1 \in \calP_2(T)$, and $M_2 \in \calP_2(S \setminus T)$, we have $M_1 \sqcup M_2 \in \calP_2(S)$. On the other hand, for every $M \in \calP_2(S)$, there are exactly $\binom{s/2}{t/2}$ choices of $T \subseteq S, M_1 \in \calP_2(T)$, and $M_2 \in \calP_2(S \setminus T)$ such that $M_1 \sqcup M_2 = M$, since $M_1$ must consist of exactly $\frac{t}{2}$ pairs of $M$, $M_2 = M \setminus M_1$, and $T$ is determined by $M_1$. Thus, we see
    \begin{align*}
        &\sum_{\substack{T \subseteq S:\\ |T| = t}} \left( \sum_{M_1 \in \calP_2(T)} \prod_{ \{i,j\} \in M_1 } \det\nolimits^{2}\begin{pmatrix}
            \uu_i^\top \\
            \uu_j^\top
        \end{pmatrix}\right)\left( \sum_{M_2 \in \calP_2(S \setminus T)} \prod_{ \{i,j\} \in M_2 } \det\nolimits^{2}\begin{pmatrix}
            \uu_i^\top \\
            \uu_j^\top
        \end{pmatrix}\right)\\
        &= \binom{\frac{s}{2}}{\frac{t}{2}} \sum_{M \in \calP_2(S)} \prod_{ \{i,j\} \in M } \det\nolimits^{2}\begin{pmatrix}
            \uu_i^\top \\
            \uu_j^\top
        \end{pmatrix}\\
        &= \binom{\frac{s}{2}}{\frac{t}{2}} \Delta_S(\UU). \qedhere
    \end{align*}
\end{proof}

Using the $\Delta$-polynomials, our goal in this section will be to prove the following statement, which gives a sum-of-squares proof of Theorem~\ref{thm:perm-inequality}.

\begin{theorem}\label{thm:perm-inequality-SoS}
    Let $\AA = \UU \VV^\top \in \bR^{n \times n}$ where $\UU, \VV \in \bR^{n \times 2}$. Then,
    \begin{align*}
        \frac{\per^2(\AA)}{(n!)^2} - \frac{\per(\AA \otimes \JJ_2)}{(2n)!} &= \sum_{k=1}^{\lfloor n/2 \rfloor} C_{n,k} \sum_{\substack{I, J \subseteq [n]:\\ |I| = |J| = n-2k }} \frac{\per^2(\AA_{I,J})}{((n-2k)!)^2} \cdot \Delta_{[n] \setminus I}(\UU) \Delta_{[n] \setminus J}(\VV),
    \end{align*}
    where the coefficients $C_{n,k}$ for $1 \le k \le \lfloor n/2\rfloor$ are given by
    \begin{align*}
        C_{n,k} &\colonequals \frac{(2n-2k)!(2k)!(n-2k)!(n-2k+1)!}{(2n)!(2k-1)(n-k)!(n-k+1)!}.
    \end{align*}
    In particular, $C_{n,k} \ge 0$ for $1 \le k \le \lfloor n/2\rfloor$, and
    \[\frac{\per^2(\AA)}{(n!)^2} - \frac{\per(\AA \otimes \JJ_2)}{(2n)!}\]
    equals a homogeneous sum-of-squares polynomial in $\UU$ and $\VV$ of degree $4n$.
\end{theorem}

Recall that the choices of $\AA = x\JJ_n - \aalpha \oone_n^\top - \oone_n \bbeta^\top$ and $\AA = x\JJ_n - \aalpha \bbeta^\top$ are used to establish the finite free convolution inequalities in Theorem~\ref{thm:ffc-inequality}. Hence we get the following corollary for the finite free convolutions using the sum-of-squares proof in Theorem~\ref{thm:perm-inequality-SoS}.
\begin{corollary}\label{cor:ffc-inequality-SoS}
    For monic real-rooted polynomials \begin{align*}
        p(x) &= \prod_{i=1}^n (x - \alpha_i), \\
        q(x) &= \prod_{i=1}^n (x - \beta_i)
    \end{align*}
    of degree $n$, let $\aalpha \colonequals (\alpha_1, \dots, \alpha_n)^\top, \bbeta \colonequals (\beta_1, \dots, \beta_n)^\top \in \bR^n$. Then, \begin{align*}
        (p \boxplus_n q)^2(x) - (p^2 \boxplus_{2n} q^2)(x) \quad \text{ and } \quad (p \boxtimes_n q)^2(x) - (p^2 \boxtimes_{2n} q^2)(x)
    \end{align*}
    are sum-of-squares polynomials in $x, \aalpha$ and $\bbeta$.
\end{corollary}

Before we prove Theorem~\ref{thm:perm-inequality-SoS}, we state a few auxiliary results. In this section, we reserve $F$ for the special class of bivariate polynomials which are products of linear polynomials, defined as follows.

\begin{definition}[$F$-polynomial]
    Let $\UU = (\uu_1, \dots, \uu_n)^\top \in \bR^{n \times 2}$ be a matrix with row vectors $\uu_1, \dots, \uu_n \in \bR^2$. For $1\le i \le n$, denote
    \begin{align*}
        \uu_i = (a_i, b_i)^\top \in \bR^2.
    \end{align*}
    We use $F_{\UU}(x,y) \in \calW_n$ to denote the bivariate polynomial defined by
    \begin{align*}
        F_{\UU}(x,y) \colonequals \prod_{i=1}^n (a_i x + b_i y).
    \end{align*}
    For $I \subseteq [n]$ of size $|I| = m$, we use $F_{\UU, I} \in \calW_{m}$ to denote the bivariate polynomial defined by
    \begin{align*}
        F_{\UU, I}(x,y) \colonequals \prod_{i\in I} (a_i x + b_i y).
    \end{align*}
\end{definition}

The following lemma expresses the permanents of rank-2 matrices using the $F$-polynomials.
\begin{lemma}\label{lem:per-inner-product}
    Let $\AA = \UU \VV^\top \in \bR^{n \times n}$ where $\UU, \VV \in \bR^{n \times 2}$. Then,
    \begin{align*}
        \langle F_{\UU}, F_{\VV} \rangle &= \per(\AA),\\
        \langle F_{\UU}^2, F_{\VV}^2 \rangle &= \per(\AA \otimes \JJ_2).
    \end{align*}
    For $I, J \subseteq [n]$ of equal size,
    \begin{align*}
        \langle F_{\UU,I}, F_{\VV,J} \rangle &= \per(\AA_{I,J}).
    \end{align*}
\end{lemma}

\begin{proof}
    Let $\UU = (\uu_1, \dots, \uu_n)^\top, \VV = (\vv_1, \dots, \vv_n)^\top \in \bR^{n \times 2}$ where
    \begin{align*}
        \uu_i = (a_i,b_i)^\top, \vv_i = (c_i,d_i)^\top \in \bR^2, \quad \text{ for }1\le i \le n.
    \end{align*}
    We directly compute
    \begin{align*}
        &\langle F_{\UU}, F_{\VV} \rangle\\
        &= \left\langle \prod_{i=1}^n (a_i x + b_i y), \prod_{i=1}^n (c_i x + d_i y)\right\rangle\\
        &= \left\langle \sum_{k=0}^n \left(\sum_{\substack{A \subseteq [n]:\\ |A| = k}} \prod_{i \in A} a_i \prod_{j \in [n] \setminus A} b_j\right) x^k y^{n-k}, \sum_{k=0}^n \left(\sum_{\substack{B \subseteq [n]:\\ |B| = k}} \prod_{i \in B} c_i \prod_{j \in [n] \setminus B} d_j\right) x^k y^{n-k}\right\rangle\\
        &= \sum_{k=0}^n k!(n-k)!\left(\sum_{\substack{A \subseteq [n]:\\ |A| = k}} \prod_{i \in A} a_i \prod_{j \in [n] \setminus A} b_j\right)\left(\sum_{\substack{B \subseteq [n]:\\ |B| = k}} \prod_{i \in B} c_i \prod_{j \in [n] \setminus B} d_j\right)\\
        &= \sum_{k=0}^n \sum_{\substack{A, B \subseteq [n]:\\ |A| = |B| = k} } \sum_{\substack{\pi \in \calS_n:\\ \pi(A) = B} } \prod_{i \in A} (a_i c_{\pi(i)}) \prod_{j \in [n] \setminus A} (b_j d_{\pi(j)})\\
        &= \sum_{\pi \in \calS_n} \prod_{i \in [n]} (a_i c_{\pi(i)} + b_i d_{\pi(i)})\\
        &= \sum_{\pi \in \calS_n} \prod_{i \in [n]} \AA_{i,\pi(i)}\\
        &= \per(\AA).
    \end{align*}
    The same proof works for $\langle F_{\UU}^2, F_{\VV}^2 \rangle$ and $\langle F_{\UU,I}, F_{\VV,J} \rangle$ by encoding them as $F$-polynomials.
\end{proof}

We also have the following identity for the transvectant of the $F$-polynomials, the proof of which is deferred to Section~\ref{sec:transvectant-identity}.
\begin{proposition}\label{prop:transvectant-identity}
    Let $\UU \in \bR^{n \times 2}$ and $0 \le k \le \lfloor n/2 \rfloor$. Then,
    \begin{align*}
        \tau_{2k}(F_{\UU}, F_{\UU}) = (-1)^k \frac{(2k)!(n-k)!}{(n-2k)!} \sum_{\substack{I \subseteq [n]:\\ |I| = n-2k}} F_{\UU, I}(x,y)^2 \cdot \Delta_{[n] \setminus I}(\UU).
    \end{align*}
\end{proposition}
\noindent
The final ingredient we will use is the following standalone combinatorial result, the proof of which is deferred to Section~\ref{sec:combinatorial-identity}.
\begin{lemma}
\label{lem:combinatorial-identity}
    Let $n, k \in \bN$ satisfy $1 \le k \le \lfloor n/2\rfloor$. Then,
    \begin{align*}
        \sum_{t=0}^k \frac{2n-4t+1}{2k-2t-1} \binom{2t}{t} \binom{2k-2t}{k-t} \frac{((n-t)!)^2(2n-2t-2k)!}{(n-t-k)!(n-t-k+1)!(2n-2t+1)!} = 0.
    \end{align*}
\end{lemma}
\noindent
With these results, we now give the proof of our main claim.
\begin{proof}[Proof of Theorem~\ref{thm:perm-inequality-SoS}]
    Let $\AA = \UU \VV^\top \in \bR^{n \times n}$ where $\UU, \VV \in \bR^{n \times 2}$. By Lemma~\ref{lem:per-inner-product}, to prove
    \begin{align*}
        \frac{\per^2(\AA)}{(n!)^2} - \frac{\per(\AA \otimes \JJ_2)}{(2n)!} &= \sum_{k=1}^{\lfloor n/2 \rfloor} C_{n,k} \sum_{\substack{I, J \subseteq [n]:\\ |I| = |J| = n-2k }} \frac{\per^2(\AA_{I,J})}{((n-2k)!)^2} \cdot \Delta_{[n] \setminus I}(\UU) \Delta_{[n] \setminus J}(\VV),
    \end{align*}
    it is enough to show
    \begin{align*}
        \frac{\langle F_{\UU}, F_{\VV} \rangle^2}{(n!)^2} - \frac{\langle F_{\UU}^2, F_{\VV}^2 \rangle}{(2n)!} &= \sum_{k=1}^{\lfloor n/2 \rfloor} C_{n,k} \sum_{\substack{I, J \subseteq [n]:\\ |I| = |J| = n-2k }} \frac{\langle F_{\UU, I}, F_{\VV, J} \rangle^2}{((n-2k)!)^2} \cdot \Delta_{[n] \setminus I}(\UU) \Delta_{[n] \setminus J}(\VV),
    \end{align*}
    or equivalently
    \begin{align}
        \frac{\langle F_{\UU}^2, F_{\VV}^2 \rangle}{(2n)!} = \frac{\langle F_{\UU}, F_{\VV} \rangle^2}{(n!)^2}  -\sum_{k=1}^{\lfloor n/2 \rfloor} C_{n,k} \sum_{\substack{I,J \subseteq [n]:\\ |I| = |J| = n-2k} } \frac{\langle F_{\UU, I}, F_{\VV, J} \rangle^2}{((n-2k)!)^2} \cdot \Delta_{[n] \setminus I}(\UU) \Delta_{[n] \setminus J}(\VV). \label{eq:inner-product-SoS}
    \end{align}

    We prove \eqref{eq:inner-product-SoS} by strong induction on $n$. We verify the base cases $n = 0$ and $n = 1$ as follows:
    \begin{align*}
        \frac{\langle 1^2, 1^2 \rangle}{0!} &= 1\\
        &= \frac{\langle 1, 1 \rangle^2}{0!},\\
        \frac{\langle (ax + by)^2, (cx + dy)^2 \rangle}{2!} &= \frac{1}{2}\langle a^2 x^2 + 2ab xy + b^2 y^2, c^2x^2 + 2cd xy + d^2 y^2 \rangle\\
        &= a^2c^2 + 2abcd + b^2d^2\\
        &= (ac + bd)^2\\
        &= \frac{\langle ax+by, cx+dy\rangle^2}{(1!)^2}.
    \end{align*}

    Now fix $n \ge 2$ and assume that \eqref{eq:inner-product-SoS} holds for all $0 \le n' < n$. We will show that \eqref{eq:inner-product-SoS} holds for $n$. By Theorem~\ref{thm:inner-product-decom} and Proposition~\ref{prop:transvectant-identity}, we have
    \begin{align*}
        \frac{\langle F_{\UU}, F_{\VV} \rangle^2}{(n!)^2}
        &= \sum_{k=0}^{\lfloor n/2 \rfloor} \frac{c_{n,k}(2n-4k)!}{(n!)^2}\cdot \frac{\langle \tau_{2k}(F_{\UU}, F_{\UU}), \tau_{2k}(F_{\VV}, F_{\VV})\rangle}{(2n-4k)!}\\
        &= \sum_{k=0}^{\lfloor n/2 \rfloor} \frac{c_{n,k}(2n-4k)!((2k)!)^2 ((n-k)!)^2 }{(n!)^2((n-2k)!)^2} \\
        &\hspace{2cm} \sum_{\substack{I, J \subseteq [n]:\\ |I| = |J| = n-2k} } \frac{\langle F_{\UU, I}(x,y)^2, F_{\VV, J}(x,y)^2 \rangle}{(2n-4k)!} \cdot \Delta_{[n] \setminus I}(\UU) \Delta_{[n] \setminus J}(\VV),
    \intertext{where $c_{n,k} = \frac{(2n-4k+1)!((n-2k)!)^2}{(2k)!(2n-2k+1)!(2n-4k)!}$. Simplifying the expression, we get}
        &= \sum_{k=0}^{\lfloor n/2 \rfloor} D_{n,k} \sum_{\substack{I, J \subseteq [n]:\\ |I| = |J| = n-2k} } \frac{\langle F_{\UU, I}(x,y)^2, F_{\VV, J}(x,y)^2 \rangle}{(2n-4k)!} \cdot \Delta_{[n] \setminus I}(\UU) \Delta_{[n] \setminus J}(\VV),
    \end{align*}
    where the coefficients $D_{n,k}$ for $0 \le k \le \lfloor n/2\rfloor$ are given by
    \begin{align*}
        D_{n,k} \colonequals \frac{(2k)!((n-k)!)^2(2n-4k+1)!}{(n!)^2(2n-2k+1)!}.
    \end{align*}
    At $k = 0$, the coefficient evaluates to $D_{n,0} = 1$. Therefore,
    \begin{align*}
        &\frac{\langle F_{\UU}, F_{\VV} \rangle^2}{(n!)^2} - \frac{\langle F_{\UU}^2, F_{\VV}^2 \rangle}{(2n)!}\\
        &\hspace{1cm} = \sum_{k=1}^{\lfloor n/2 \rfloor} D_{n,k} \sum_{\substack{I, J \subseteq [n]:\\ |I| = |J| = n-2k} } \frac{\langle F_{\UU, I}(x,y)^2, F_{\VV, J}(x,y)^2 \rangle}{(2n-4k)!} \cdot \Delta_{[n] \setminus I}(\UU) \Delta_{[n] \setminus J}(\VV) \numberthis \label{eq:inner-product-expansion}
    \end{align*}
    Note that for $1\le k \le \lfloor n/2\rfloor$ and $I, J \subseteq [n]$ of size $n-2k$, by the inductive hypothesis we have
    \begin{align*}
        &\frac{\langle F_{\UU, I}(x,y)^2, F_{\VV, J}(x,y)^2 \rangle}{(2n-4k)!}\\
        &\hspace{0.5cm} = \frac{\langle F_{\UU, I}, F_{\VV, J} \rangle^2}{((n-2k)!)^2} - \sum_{\ell = 1}^{\lfloor \frac{n-2k}{2} \rfloor} C_{n-2k,\ell} \sum_{\substack{I' \subseteq I ,J' \subseteq J:\\ |I'| = |J'| = n-2k - 2\ell} } \frac{\langle F_{\UU, I'}, F_{\VV, J'} \rangle^2}{((n-2k-2\ell)!)^2} \cdot \Delta_{I \setminus I'}(\UU) \Delta_{J \setminus J'}(\VV) \numberthis \label{eq:inductive-hypothesis-replacement}
    \end{align*}
    Substituting the summands in \eqref{eq:inner-product-expansion} with \eqref{eq:inductive-hypothesis-replacement}, we get
    \begin{align*}
        &\frac{\langle F_{\UU}, F_{\VV} \rangle^2}{(n!)^2} - \frac{\langle F_{\UU}^2, F_{\VV}^2 \rangle}{(2n)!}\\
        &\hspace{0.5cm} = \sum_{k=1}^{\lfloor n/2 \rfloor} D_{n,k} \sum_{\substack{I, J \subseteq [n]:\\ |I| = |J| = n-2k} } \frac{\langle F_{\UU, I}(x,y)^2, F_{\VV, J}(x,y)^2 \rangle}{(2n-4k)!} \cdot \Delta_{[n] \setminus I}(\UU) \Delta_{[n] \setminus J}(\VV)\\
        &\hspace{0.5cm} = G_n(\UU, \VV) - H_n(\UU, \VV),
    \end{align*}
    where the two parts $G_n$ and $H_n$ are given by
    \begin{align*}
        G_n(\UU, \VV) &\colonequals \sum_{k=1}^{\lfloor n/2 \rfloor} D_{n,k} \sum_{\substack{I, J \subseteq [n]:\\ |I| = |J| = n-2k} } \frac{\langle F_{\UU, I}, F_{\VV, J} \rangle^2}{((n-2k)!)^2} \cdot \Delta_{[n] \setminus I}(\UU) \Delta_{[n] \setminus J}(\VV),\\
        H_n(\UU, \VV) &\colonequals \sum_{k=1}^{\lfloor n/2 \rfloor} D_{n,k} \sum_{\substack{I, J \subseteq [n]:\\ |I| = |J| = n-2k} } \sum_{\ell = 1}^{\lfloor \frac{n-2k}{2} \rfloor} C_{n-2k,\ell} \\
        &\hspace{1cm} \sum_{\substack{I' \subseteq I, J' \subseteq J:\\ |I'| = |J'| = n-2k - 2\ell} } \frac{\langle F_{\UU, I'}, F_{\VV, J'} \rangle^2}{((n-2k-2\ell)!)^2} \Delta_{I \setminus I'}(\UU) \Delta_{J \setminus J'}(\VV) \cdot \Delta_{[n] \setminus I}(\UU) \Delta_{[n] \setminus J}(\VV).
    \end{align*}
    We further simplify $H_n$ as follows.
    \begin{align*}
        &H_n(\UU, \VV)\\
        &= \sum_{k=1}^{\lfloor n/2 \rfloor} D_{n,k} \sum_{\substack{I, J \subseteq [n]:\\ |I| = |J| = n-2k} } \sum_{\ell = 1}^{\lfloor \frac{n-2k}{2} \rfloor} C_{n-2k,\ell} \\
        &\quad \sum_{\substack{I' \subseteq I, J' \subseteq J:\\ |I'| = |J'| = n-2k - 2\ell} } \frac{\langle F_{\UU, I'}, F_{\VV, J'} \rangle^2}{((n-2k-2\ell)!)^2} \Delta_{I \setminus I'}(\UU) \Delta_{J \setminus J'}(\VV) \cdot \Delta_{[n] \setminus I}(\UU) \Delta_{[n] \setminus J}(\VV)\\
        &= \sum_{k=1}^{\lfloor n/2 \rfloor} D_{n,k} \sum_{\ell = 1}^{\lfloor \frac{n-2k}{2} \rfloor} C_{n-2k,\ell} \sum_{\substack{I',J' \subseteq [n]:\\ |I'| = |J'| = n-2k - 2\ell} } \frac{\langle F_{\UU, I'}, F_{\VV, J'} \rangle^2}{((n-2k-2\ell)!)^2} \\
        &\quad  \sum_{\substack{I, J \subseteq [n]:\\ I \supseteq I', J \supseteq J',\\ |I| = |J| = n-2k} } \Delta_{I \setminus I'}(\UU) \Delta_{J \setminus J'}(\VV) \Delta_{[n] \setminus I}(\UU) \Delta_{[n] \setminus J}(\VV)
        \intertext{Then, using Lemma~\ref{lem:delta-contraction},}
        &= \sum_{k=1}^{\lfloor n/2 \rfloor} D_{n,k} \sum_{\ell = 1}^{\lfloor \frac{n-2k}{2} \rfloor} C_{n-2k,\ell} \sum_{\substack{I',J' \subseteq [n]:\\ |I'| = |J'| = n-2k - 2\ell} } \frac{\langle F_{\UU, I'}, F_{\VV, J'} \rangle^2}{((n-2k-2\ell)!)^2} \binom{k+\ell}{k}^2 \Delta_{[n] \setminus I'}(\UU) \Delta_{[n] \setminus J'}(\VV) \\
        &= \sum_{k=1}^{\lfloor n/2\rfloor} \sum_{t=1}^{k-1} D_{n,t} C_{n-2t,k-t} \binom{k}{t}^2 \sum_{\substack{I, J \subseteq [n]:\\ |I| = |J| = n-2k} } \frac{\langle F_{\UU, I}, F_{\VV, J} \rangle^2}{((n-2k)!)^2} \cdot \Delta_{[n] \setminus I}(\UU) \Delta_{[n] \setminus J}(\VV).
    \end{align*}
    Therefore, we have
    \begin{align*}
        &\frac{\langle F_{\UU}, F_{\VV} \rangle^2}{(n!)^2} - \frac{\langle F_{\UU}^2, F_{\VV}^2 \rangle}{(2n)!}\\
        &= G_n(\UU, \VV) - H_n(\UU, \VV)\\
        &= \sum_{k=1}^{\lfloor n/2\rfloor} \left(D_{n,k} - \sum_{t=1}^{k-1} D_{n,t} C_{n-2t,k-t} \binom{k}{t}^2\right) \sum_{\substack{I, J \subseteq [n]:\\ |I| = |J| = n-2k} } \frac{\langle F_{\UU, I}, F_{\VV, J} \rangle^2}{((n-2k)!)^2} \cdot \Delta_{[n] \setminus I}(\UU) \Delta_{[n] \setminus J}(\VV).
    \end{align*}
    To show that \eqref{eq:inner-product-SoS} holds for $n$, it suffices to prove
    \begin{align}
        C_{n,k} = D_{n,k} - \sum_{t=1}^{k-1} D_{n,t} C_{n-2t,k-t} \binom{k}{t}^2, \quad \text{ for } 1 \le k \le \lfloor n/2\rfloor. \label{eq:coefficient-identity}
    \end{align}
    Recall that
    \begin{align*}
        C_{n,k} &= \frac{(2n-2k)!(2k)!(n-2k)!(n-2k+1)!}{(2n)!(2k-1)(n-k)!(n-k+1)!}, \\
        D_{n,k} &= \frac{(2k)!((n-k)!)^2(2n-4k+1)!}{(n!)^2(2n-2k+1)!}
    \end{align*}
    for $1\le k \le \lfloor n/2\rfloor$. Using these expressions, we may extend the definition of $C_{n,k}$ and $D_{n,k}$ to $k = 0$ by setting $C_{n,0} \colonequals -1$ and $D_{n, 0} \colonequals 1$. The relation \eqref{eq:coefficient-identity} is then equivalent to
    \begin{align*}
        \sum_{t=0}^k D_{n,t} C_{n-2t,k-t} \binom{k}{t}^2 = 0, \quad \text{ for } 1 \le k \le \lfloor n/2\rfloor.
    \end{align*}
    After plugging in the expressions for $C_{n,k}$ and $D_{n,k}$, we can rewrite the summands as
    \begin{align*}
        &D_{n,t} C_{n-2t,k-t} \binom{k}{t}^2\\
        &= \frac{(2t)!((n-t)!)^2(2n-4t+1)!}{(n!)^2(2n-2t+1)!} \cdot \frac{(2n-2k-2t)!(2k-2t)!(n-2k)!(n-2k+1)!}{(2n-4t)!(2k-2t-1)(n-k-t)!(n-k-t+1)!} \cdot \\
        &\hspace{3cm} \cdot \frac{(k!)^2}{(t!)^2((k-t)!)^2}\\
        &= \frac{(n-2k)!(n-2k+1)!(k!)^2}{(n!)^2} \cdot \frac{2n-4t+1}{2k-2t-1} \cdot\frac{(2t)!}{(t!)^2} \cdot\frac{(2k-2t)!}{((k-t)!)^2} \cdot \\ &\hspace{3cm} \cdot \frac{((n-t)!)^2(2n-2k-2t)!}{(n-k-t)!(n-k-t+1)!(2n-2t+1)!}\\
        &= K_{n,k} \Phi_{n,k}(t),
    \end{align*}
    where we denote for $1 \le k \le \lfloor n/2\rfloor$,
    \begin{align*}
        K_{n,k} &\colonequals \frac{(n-2k)!(n-2k+1)!(k!)^2}{(n!)^2},\\
        \Phi_{n,k}(t) &\colonequals \frac{2n-4t+1}{2k-2t-1} \binom{2t}{t} \binom{2k-2t}{k-t} \frac{((n-t)!)^2(2n-2t-2k)!}{(n-t-k)!(n-t-k+1)!(2n-2t+1)!}, \quad \text{ for } 0 \le t \le k.
    \end{align*}
    By Lemma~\ref{lem:combinatorial-identity}, for $1 \le k \le \lfloor n/2\rfloor$,
    \begin{align*}
        \sum_{t=0}^k \Phi_{n,k}(t) = 0.
    \end{align*}
    As a result, \eqref{eq:coefficient-identity} holds since
    \begin{align*}
        \sum_{t=0}^k D_{n,t} C_{n-2t,k-t} \binom{k}{t}^2 &= K_{n,k} \sum_{t=0}^k \Phi_{n,k}(t) = 0,
    \end{align*}
    and we conclude that \eqref{eq:inner-product-SoS} holds for $n$, completing the inductive step.

    By induction, \eqref{eq:inner-product-SoS} holds for all $n \in \bN$, which shows that
    \begin{align*}
        \frac{\per^2(\AA)}{(n!)^2} - \frac{\per(\AA \otimes \JJ_2)}{(2n)!} &= \sum_{k=1}^{\lfloor n/2 \rfloor} C_{n,k} \sum_{\substack{I, J \subseteq [n]:\\ |I| = |J| = n-2k }} \frac{\per^2(\AA_{I,J})}{((n-2k)!)^2} \cdot \Delta_{[n] \setminus I}(\UU) \Delta_{[n] \setminus J}(\VV). \qedhere
    \end{align*}
\end{proof}

\section{Free Multiplicative Convolutions: Proof of Proposition~\ref{prop:multiplicative-finite-to-free-endpoint}}
\label{sec:mult-endpoint}

\begin{proof}
By commutativity, we may assume that $q$ has only non-negative roots.
The real-rootedness assertion is well-known; see, for example, Proposition~2.16 of \cite{Fujie-2025-RegularityConvergenceFiniteFreeConvolutions}.

First, suppose that the roots $\beta_1,\ldots,\beta_n$ of $q$ are strictly positive, and write $\alpha_1,\ldots,\alpha_n$ for the roots of $p$.
For a polynomial $f(x)$ of degree $n$, define
\[ f^{\vee}(y) \colonequals (-1)^n f(-y). \]
This operation is an involution of polynomials, satisfies $\maxroot(f^{\vee}) = -\minroot(f)$ and, for $f, g$ two polynomials of degree $n$, satisfies $(f \boxplus_n g)^{\vee} = f^{\vee} \boxplus_n g^{\vee}$.

For a fixed $x \in \bR$, define
\begin{align*}
    c_x(y) &\colonequals \prod_{i=1}^n \left(y-\frac{x}{\beta_i}\right), \\
    h_x &\colonequals c_x \boxplus_n p^{\vee}
\end{align*}
The root formulas of Proposition~\ref{prop:real-rooted-ffc} give
\begin{align}
    (p \boxtimes_n q)(x)
    =(-1)^n\left(\prod_{i=1}^n \beta_i\right)h_x(0).
    \label{eq:multiplicative-to-additive-endpoint}
\end{align}

Let $a$ and $b$ be free bounded self-adjoint operators with laws $\mu_p$ and $\mu_q$, respectively, in a tracial $C^*$-probability space with faithful trace.
Then, by the definition of free multiplicative convolution,
\begin{align*}
    E \colonequals \maxsupp(\mu_p \boxtimes \mu_q)
    = \max \spec(b^{1/2}ab^{1/2}).
\end{align*}
For every $x>E$,
\begin{align*}
    xb^{-1}-a
    =b^{-1/2}\left(xI-b^{1/2}ab^{1/2}\right)b^{-1/2}>0.
\end{align*}
The law of $xb^{-1}-a$ is $\mu_{c_x}\boxplus\mu_{p^{\vee}}$.
Observing that $h_x^\vee=c_x^\vee\boxplus_n(p^\vee)^\vee = c_x^{\vee} \boxplus_n p$ and applying Proposition~\ref{prop:additive-finite-to-free-endpoint} to the reflected polynomials shows that
\begin{align*}
    \minroot(h_x) \geq \minsupp(\mu_{c_x}\boxplus\mu_{p^{\vee}})>0.
\end{align*}
It follows that $(-1)^n h_x(0)>0$, and hence, by \eqref{eq:multiplicative-to-additive-endpoint},
\begin{align*}
    (p \boxtimes_n q)(x)>0
\end{align*}
for every $x>E$.
This proves the claimed bound when the roots of $q$ are strictly positive.

For the general case, set $q_{\epsilon}(y)\colonequals q(y-\epsilon)$ for $\epsilon>0$.
The result just proved applies to $p$ and $q_{\epsilon}$.
As $\epsilon \to 0$, the polynomials $p\boxtimes_n q_{\epsilon}$ converge coefficientwise to $p\boxtimes_n q$, and hence their largest roots converge.
For the free operators, we also have
\begin{align*}
    (b+\epsilon I)^{1/2}a(b+\epsilon I)^{1/2}
    \longrightarrow b^{1/2}ab^{1/2}
\end{align*}
in operator norm, so $\maxsupp$ values of the corresponding free multiplicative convolutions converge as well.
The result then follows by letting $\epsilon \to 0$.
\end{proof}

\section{Transvectant Identity: Proof of Proposition~\ref{prop:transvectant-identity}} \label{sec:transvectant-identity}

We first encode the transvectant $\tau_{2k}(F_{\UU}, F_{\UU})$ using the permanent of a matrix, as follows.

\begin{proposition} \label{prop:transvectant-permanent-identity}
Let $\UU = (\uu_1, \dots, \uu_n)^\top \in \bR^{n \times 2}$ be a matrix with row vectors $\uu_1, \dots, \uu_n \in \bR^2$, and $0 \le k \le \lfloor n/2\rfloor$. For $n+1 \le i \le 2n-2k$, define
\begin{align*}
    \uu_i \colonequals (-y, x)^\top.
\end{align*}
Construct the $(2n-2k) \times (2n-2k)$ matrix $\CC$ whose entries are given by
\begin{align*}
    \CC_{i,j} \colonequals \det\begin{pmatrix}
        \uu_i^\top\\
        \uu_j^\top
    \end{pmatrix}, \quad \text{ for } (i,j) \in [2n-2k]^2.
\end{align*}
Then,
\begin{align*}
    \tau_{2k}(F_{\UU}, F_{\UU}) = (-1)^{n-2k}\frac{(2k)!}{((n-2k)!)^2} \cdot \per(\CC).
\end{align*}
\end{proposition}

\begin{proof}
    For $1\le i \le n$, denote
    \begin{align*}
        \uu_i = (a_i, b_i)^\top \in \bR^2.
    \end{align*}
    By the definition of transvectant,
    \begin{align*}
        \tau_{2k}(F_{\UU}, F_{\UU}) &= \calR(\Omega^{2k} (F_{\UU}(x_1, y_1) F_{\UU}(x_2, y_2)))\\
        &= \calR\left(\Omega^{2k} \left(\prod_{i \in [n]}(a_i x_1 + b_i y_1)\prod_{j \in [n]}(a_j x_2 + b_j y_2)\right)\right)\\
        \intertext{Expanding the $2k$ applications of $\Omega$ and grouping terms by the sets of differentiated factors gives}
        &= \calR\left((2k)! \sum_{\substack{A, B \subseteq [n]:\\ |A|=|B|=2k}} \per(\CC_{A,B}) \prod_{i \in [n] \setminus A }(a_i x_1 + b_i y_1)\prod_{j \in [n] \setminus B}(a_j x_2 + b_j y_2) \right)\\
        &= (2k)! \sum_{\substack{A, B \subseteq [n]:\\ |A|=|B|=2k}} \per(\CC_{A,B}) \prod_{i \in [n] \setminus A }(a_i x + b_i y)\prod_{j \in [n] \setminus B}(a_j x + b_j y). \numberthis \label{eq:perC-1}
    \end{align*}
    For fixed $A$ and $B$, each matching between their elements occurs once for each of the $(2k)!$ possible orders in which its pairs can be differentiated, which gives the permanent and the displayed factorial.
    On the other hand,
    \begin{align*}
        \per(\CC) = \sum_{\pi \in \calS_{2n-2k}} \prod_{i=1}^{2n-2k} \CC_{i,\pi(i)}.
    \end{align*}
    By definition of $\CC$, we have
    \begin{align*}
        \CC_{i,j} = \det\begin{pmatrix}
            \uu_i^\top\\
            \uu_j^\top
        \end{pmatrix} &= \begin{cases}
            a_ib_j - a_jb_i & \quad \text{ if } 1\le i \le n, 1\le j\le n,\\
            a_i x + b_i y & \quad \text{ if } 1 \le i \le n, n+1 \le j \le 2n-2k,\\
            -a_jx - b_j y & \quad \text{ if } n+1\le i \le 2n-2k, 1\le j \le n,\\
            0 & \quad \text{ if } n+1 \le i \le 2n-2k, n+1\le j \le 2n-2k.
        \end{cases}
    \end{align*}
    Since $\CC_{i,j} = 0$ when both $i$ and $j$ are at least $n+1$, the only permutations $\pi \in \calS_{2n-2k}$ that contribute to $\per(\CC)$ satisfy $\pi(i) \in [n]$ and $\pi^{-1}(i) \in [n]$ for $n+1\le i \le 2n-2k$. Thus,
    \begin{align*}
        \per(\CC) &= \sum_{\pi \in \calS_{2n-2k}} \prod_{i=1}^{2n-2k} \CC_{i,\pi(i)}\\
        &= \sum_{\substack{A,B \subseteq [n]:\\ |A|=|B|=2k} } \;\sum_{\substack{\pi \in \calS_{2n-2k}: \text{ for}\\  n+1\le i \le 2n-2k,\\ \pi(i) \in [n] \setminus A,\\ \pi^{-1}(i) \in [n] \setminus B} } \prod_{i=1}^{2n-2k} \CC_{i,\pi(i)}\\
        &= (-1)^{n-2k}((n-2k)!)^2 \sum_{\substack{A,B \subseteq [n]:\\ |A|=|B|=2k} }\per(\CC_{A,B}) \prod_{i \in [n] \setminus A }(a_i x + b_i y)\prod_{j \in [n] \setminus B}(a_j x + b_j y). \numberthis \label{eq:perC-2}
    \end{align*}
    Comparing \eqref{eq:perC-1} and \eqref{eq:perC-2}, we conclude
    \begin{align*}
        \tau_{2k}(F_{\UU}, F_{\UU}) &= (-1)^{n-2k}\frac{(2k)!}{((n-2k)!)^2} \cdot \per(\CC). \qedhere
    \end{align*}
\end{proof}

We will also use the following corollary of MacMahon's Master Theorem to express the permanent in Proposition~\ref{prop:transvectant-permanent-identity}.

\begin{theorem}[{\cite[p.~17]{percus2012combinatorial}}] \label{thm:MMT-corollary}
    Let $\AA \in \bR^{m \times m}$, and $\TT \colonequals \diag(t_1, \dots, t_m)$ where $t_1, \dots, t_m$ are formal variables. Then,
    \begin{align*}
        \per(\AA) = [t_1\cdots t_m] \frac{1}{\det(\II_m - \TT \AA)}.
    \end{align*}
\end{theorem}

Now, we are ready to finish the proof of Proposition~\ref{prop:transvectant-identity}.

\begin{proof}[Proof of Proposition~\ref{prop:transvectant-identity}]
    Let $\UU = (\uu_1, \dots, \uu_n)^\top \in \bR^{n \times 2}$ be a matrix with row vectors $\uu_1, \dots, \uu_n \in \bR^2$, and $0 \le k \le \lfloor n/2\rfloor$. Recall that for $n+1 \le i \le 2n-2k$, we define
    \begin{align*}
        \uu_i \colonequals (-y, x)^\top,
    \end{align*}
    and create a $(2n-2k) \times (2n-2k)$ matrix $\CC$ whose entries are given by
    \begin{align*}
        \CC_{i,j} \colonequals \det\begin{pmatrix}
            \uu_i^\top\\
            \uu_j^\top
        \end{pmatrix}, \quad \text{ for } (i,j) \in [2n-2k]^2.
    \end{align*}

    By Proposition~\ref{prop:transvectant-permanent-identity},
    \begin{align}
        \tau_{2k}(F_{\UU}, F_{\UU}) = (-1)^{n-2k}\frac{(2k)!}{((n-2k)!)^2} \cdot \per(\CC). \label{eq:transvectant-permanent-identity}
    \end{align}
    Using Theorem~\ref{thm:MMT-corollary}, we have
    \begin{align*}
        \per(\CC) = [t_1 \dots t_{2n-2k}] \frac{1}{\det(\II_{2n-2k} - \TT \CC)},
    \end{align*}
    where $\TT = \diag(t_1, \dots, t_{2n-2k})$ is a diagonal matrix of formal variables.

    Writing
    \begin{align*}
        \begin{pmatrix}
            \uu_1^\top\\
            \uu_2^\top\\
            \vdots\\
            \uu_{2n-2k}^\top
        \end{pmatrix} \equalscolon \begin{pmatrix}
            \aa & \bb
        \end{pmatrix}
    \end{align*}
    where $\aa = (a_1, \dots, a_{2n-2k})^\top$ and $\bb = (b_1, \dots, b_{2n-2k})^\top$, we note that
    \begin{align*}
        \CC_{i,j} &= \det\begin{pmatrix}
            \uu_i^\top\\
            \uu_j^\top
        \end{pmatrix}\\
        &= a_i b_j - b_i a_j,
    \end{align*}
    and thus $\CC = \aa \bb^\top - \bb \aa^\top$. Then, using Sylvester's identity,
    \begin{align*}
        \det(\II_{2n-2k} - \TT \CC)
        &= \det\left(\II_{2n-2k} - \begin{pmatrix}
            \TT \aa & -\TT \bb
        \end{pmatrix} \begin{pmatrix}
            \bb^\top \\
            \aa^\top
        \end{pmatrix}\right)\\
        &= \det\left(\II_2 - \begin{pmatrix}
            \bb^\top \\
            \aa^\top
        \end{pmatrix}\begin{pmatrix}
            \TT \aa & -\TT \bb
        \end{pmatrix}\right) \\
        &= \det\begin{pmatrix}
            1 - \sum_{i=1}^{2n-2k} t_i a_ib_i & \sum_{i=1}^{2n-2k} t_i b_i^2\\
            -\sum_{i=1}^{2n-2k} t_i a_i^2 & 1 + \sum_{i=1}^{2n-2k} t_i a_ib_i
        \end{pmatrix}\\
        &= 1 + \left(\sum_{i=1}^{2n-2k} t_i a_i^2\right)\left(\sum_{i=1}^{2n-2k} t_i b_i^2\right) - \left(\sum_{i=1}^{2n-2k} t_i a_ib_i\right)^2\\
        &= 1 + \sum_{1 \le i < j \le 2n-2k} t_i t_j (a_i^2 b_j^2 + a_j^2 b_i^2 - 2a_ib_ia_jb_j)\\
        &= 1 + \sum_{1 \le i < j \le 2n-2k} t_i t_j \CC_{i,j}^2,
    \end{align*}
    and we get
    \begin{align*}
         \per(\CC) &= [t_1 \dots t_{2n-2k}] \frac{1}{\det(\II_{2n-2k} - \TT \CC)}\\
         &= [t_1 \dots t_{2n-2k}]\left(1 + \sum_{1 \le i < j \le 2n-2k} t_i t_j \CC_{i,j}^2\right)^{-1}\\
         &= [t_1 \dots t_{2n-2k}]\sum_{\ell = 0}^\infty (-1)^{\ell} \left(\sum_{1 \le i < j \le 2n-2k} t_i t_j \CC_{i,j}^2\right)^{\ell}\\
         &= (-1)^{n-k} (n-k)! \sum_{M \in \calP_2([2n-2k])} \prod_{\{i,j\} \in M} \CC_{i,j}^2.
    \end{align*}
    Using that
    \begin{align*}
        \CC_{i,j} = \det\begin{pmatrix}
            \uu_i^\top\\
            \uu_j^\top
        \end{pmatrix} &= \begin{cases}
            a_ib_j - a_jb_i & \quad \text{ if } 1\le i \le n, 1\le j\le n,\\
            a_i x + b_i y & \quad \text{ if } 1 \le i \le n, n+1 \le j \le 2n-2k,\\
            -a_jx - b_j y & \quad \text{ if } n+1\le i \le 2n-2k, 1\le j \le n,\\
            0 & \quad \text{ if } n+1 \le i \le 2n-2k, n+1\le j \le 2n-2k,
        \end{cases}
    \end{align*}
    we obtain
    \begin{align*}
         \per(\CC) &=  (-1)^{n-k} (n-k)! \sum_{M \in \calP_2([2n-2k])} \prod_{\{i,j\} \in M} \CC_{i,j}^2\\
         &= (-1)^{n-k} (n-k)!(n-2k)! \sum_{\substack{I \subseteq [n]:\\ |I| = n-2k}} \prod_{i \in I} (a_ix + b_iy)^2 \sum_{M \in \calP_2([n] \setminus I)} \prod_{\{i,j\} \in M} \det\nolimits^{2}\begin{pmatrix}
            \uu_i^\top\\
            \uu_j^\top
        \end{pmatrix}\\
        &= (-1)^{n-k} (n-k)!(n-2k)! \sum_{\substack{I \subseteq [n]:\\ |I| = n-2k}} F_{\UU, I}(x,y)^2 \cdot \Delta_{[n]\setminus I}(\UU). \numberthis \label{eq:permanent-delta-identity}
    \end{align*}
    Combining \eqref{eq:transvectant-permanent-identity} and \eqref{eq:permanent-delta-identity}, we conclude
    \begin{align*}
        \tau_{2k}(F_{\UU}, F_{\UU}) &= (-1)^k\frac{(2k)!(n-k)!}{(n-2k)!} \sum_{\substack{I \subseteq [n]:\\ |I| = n-2k}} F_{\UU, I}(x,y)^2 \cdot \Delta_{[n]\setminus I}(\UU). \qedhere
    \end{align*}
\end{proof}

\section{Combinatorial Identity: Proof of Lemma~\ref{lem:combinatorial-identity}}
\label{sec:combinatorial-identity}

The identity in question can be proved using the Rogers–Dougall ``Very Well-Poised Sum'' \cite{dougall1906vandermonde}.

\begin{proof}
    For $n \in \bN$, we denote the $n$th rising factorial of $a$ by
    \begin{align*}
        a^{(n)} \colonequals a(a+1)\dots (a+n-1),
    \end{align*}
    and set $a^{(0)} \colonequals 1$. Suppose $p, q \in \bN$, and $a_1, \dots, a_p$, and $b_1, \dots, b_q$ satisfy the following:
    \begin{itemize}
        \item $a_p$ is a nonpositive integer.
        \item $b_i^{(r)} \ne 0$ for $1 \le i \le q$, where $r = -a_p$.
    \end{itemize}
    Define the terminating generalized hypergeometric series by
    \begin{align*}
        {}_pF_q\begin{pmatrix}
        \begin{array}{c}
             a_1, \dots, a_p  \\
             b_1, \dots, b_q
        \end{array};  z
    \end{pmatrix} &\colonequals \sum_{t=0}^{-a_p} \frac{a_1^{(t)}\dots a_p^{(t)}}{b_1^{(t)}\dots b_q^{(t)}} \cdot \frac{z^t}{t!}.
    \end{align*}
    Note that the parameters $a_1, \dots, a_p$ and $b_1, \dots, b_q$ are determined by the ratio of the consecutive summands:
    \begin{align*}
        \left(\frac{a_1^{(t+1)}\dots a_p^{(t+1)}}{b_1^{(t+1)}\dots b_q^{(t+1)}} \cdot \frac{z^{t+1}}{(t+1)!}\right)\bigg/\left(\frac{a_1^{(t)}\dots a_p^{(t)}}{b_1^{(t)}\dots b_q^{(t)}} \cdot \frac{z^t}{t!}\right) = \frac{(t+a_1)\dots (t+a_p)}{(t+b_1) \dots (t+b_q)}\cdot \frac{z}{t+1}.
    \end{align*}

    Now we rewrite our summation as a terminating ${}_5F_4$ hypergeometric series evaluated at 1. For $0 \le t \le k$, denote
    \begin{align*}
        \Phi_{n,k}(t) \colonequals \frac{2n-4t+1}{2k-2t-1} \binom{2t}{t} \binom{2k-2t}{k-t} \frac{((n-t)!)^2(2n-2t-2k)!}{(n-t-k)!(n-t-k+1)!(2n-2t+1)!}.
    \end{align*}
    For $0 \le t < k$, the ratio of the consecutive terms in the summation is
    \begin{align*}
        \frac{\Phi_{n,k}(t+1)}{\Phi_{n,k}(t)} &=\frac{\frac{2n-4t-3}{2k-2t-3} \binom{2t+2}{t+1} \binom{2k-2t-2}{k-t-1} \frac{((n-t-1)!)^2(2n-2t-2k-2)!}{(n-t-k-1)!(n-t-k)!(2n-2t-1)!}}{\frac{2n-4t+1}{2k-2t-1} \binom{2t}{t} \binom{2k-2t}{k-t} \frac{((n-t)!)^2(2n-2t-2k)!}{(n-t-k)!(n-t-k+1)!(2n-2t+1)!}}\\
        &= \frac{\left(t+\frac{3}{4}-\frac{n}{2}\right)\left(t +\frac{1}{2}-k\right)}{\left(t - \frac{1}{4} - \frac{n}{2}\right)\left(t + \frac{3}{2}-k\right)} \cdot \frac{4(t+1)\left(t+\frac{1}{2}\right)}{(t+1)^2} \cdot \frac{(t-k)^2}{4(t-k)\left(t + \frac{1}{2} -k\right)}\\
        &\quad \cdot \frac{4(t+k-n)(t-1+k-n)\left(t - \frac{1}{2}-n\right)\left(t-n\right)}{4(t-n)^2(t+k-n)\left(t+\frac{1}{2} + k -n\right)}\\
        &= \frac{\left(t - \frac{1}{2} - n\right)\left(t + \frac{3}{4} - \frac{n}{2}\right)\left(t + \frac{1}{2}\right)\left(t -1 + k - n\right)\left(t-k\right)}{\left(t - \frac{1}{4} - \frac{n}{2}\right)\left(t-n\right)\left(t + \frac{3}{2} - k\right)\left(t + \frac{1}{2} + k - n\right)} \cdot \frac{1}{t+1}.
    \end{align*}
    Therefore, we have
    \begin{align*}
        \sum_{t=0}^k \Phi_{n,k}(t) &= \Phi_{n,k}(0) \sum_{t=0}^{k} \frac{\left(-\frac{1}{2}-n\right)^{(t)}\left(\frac{3}{4} - \frac{n}{2}\right)^{(t)}\left(\frac{1}{2}\right)^{(t)}\left(-1+k-n\right)^{(t)}\left(-k\right)^{(t)}}{\left(-\frac{1}{4} -\frac{n}{2}\right)^{(t)}\left(-n\right)^{(t)} \left(\frac{3}{2}-k\right)^{(t)}\left(\frac{1}{2}+k-n\right)^{(t)}} \cdot \frac{1}{t!}\\
        &= \Phi_{n,k}(0) \cdot {}_5F_4\begin{pmatrix}
        \begin{array}{ccccc}
             -\frac{1}{2} - n & \frac{3}{4} - \frac{n}{2} & \frac{1}{2} & -1 + k - n & -k \\
             & -\frac{1}{4} - \frac{n}{2} & -n & \frac{3}{2} - k & \frac{1}{2} + k - n
        \end{array};  1
    \end{pmatrix},
    \end{align*}
    where the terminating ${}_5F_4$ is well-defined since $\left(-\frac{1}{4} - \frac{n}{2}\right)^{(k)}$, $\left(-n\right)^{(k)}$, $\left(\frac{3}{2} - k\right)^{(k)}$, and $\left(\frac{1}{2} +k-n\right)^{(k)}$ are nonzero. Applying the Rogers–Dougall Very Well-Poised Sum (see \cite[\S4.3(3)]{bailey1935generalized} and  \cite[Equation (9)]{dougall1906vandermonde}) which states that for $m \in \bN$,
    \begin{align*}
        {}_5F_4\begin{pmatrix}
        \begin{array}{ccccc}
             a & 1 + \frac{1}{2}a & b & c & -m \\
             & \frac{1}{2}a & 1+a-b & 1+a-c & 1+a+m
        \end{array};  1
    \end{pmatrix} = \frac{(1+a)^{(m)}(1+a-b-c)^{(m)}}{(1+a-b)^{(m)}(1+a-c)^{(m)}},
    \end{align*}
    we conclude that
    \begin{align*}
        \sum_{t=0}^k \Phi_{n,k}(t) &= \Phi_{n,k}(0) \cdot {}_5F_4\begin{pmatrix}
        \begin{array}{ccccc}
             -\frac{1}{2} - n & \frac{3}{4} - \frac{n}{2} & \frac{1}{2} & -1 + k - n & -k \\
             & -\frac{1}{4} - \frac{n}{2} & -n & \frac{3}{2} - k & \frac{1}{2} + k - n
        \end{array};  1
    \end{pmatrix}\\
        &= \Phi_{n,k}(0) \cdot \frac{\left(\frac{1}{2} -n\right)^{(k)}\left(1-k\right)^{(k)}}{\left(-n\right)^{(k)}\left(\frac{3}{2}-k\right)^{(k)}}\\
        &= 0,
    \end{align*}
    since $(1-k)^{(k)} = (1-k)\cdot (2-k) \cdots (-1) \cdot 0 = 0$ for $k \ge 1$.
\end{proof}

}

\bibliographystyle{alpha}
\bibliography{main}

\end{document}